\documentclass[aps,epsfig,floatfix]{revtex4-1}

\pdfoutput=1

\usepackage{titlesec}
\usepackage[latin1]{inputenc}
\usepackage{enumerate}
\usepackage{graphicx}
\usepackage{subfigure}
\usepackage{rotating}
\usepackage{amsmath,amsthm,amsmath,amsfonts,amssymb,color,graphicx}
\usepackage[usenames,dvipsnames]{pstricks}
\usepackage{epsfig}
\usepackage{pst-grad}
\usepackage{pst-plot}
\usepackage[final]{pdfpages}

\newtheorem{theorem}{Theorem}
\newtheorem{proposition}[theorem]{Proposition}
\newtheorem{lemma}[theorem]{Lemma}
\newtheorem{corollary}[theorem]{Corollary}
\newtheorem{conjecture}[theorem]{Conjecture}
\newtheorem{definition}{Definition}
\newtheorem*{remark}{Remark}
\newtheorem*{claim}{Claim}

\def\ZZ{\mathbb{Z}}
\def\RR{\mathbb{R}}

\def\NN{\mathbb{N}}

\newcommand{\Sym}[1]{\mathbf S_{#1}}

\def\Id{\mathbf{1}}
\def\Aut{\textbf{Aut}}

\def\Mtt{\cal{M}_{\rm 2\times3}}
\def\Rtt{\mathbf{R_{2\times3}}}
\def\Sfive{\Sym{5}}
\def\Ssix{\Sym{6}}

\def\Smn{\Sym{mn}}
\def\epo{\mathfrak{E}}
\def\prob{\mathbf{p}}
\def\fd{\mathfrak{D}}
\def\fdM{\mathfrak{M}}
\def\btr{\blacktriangleright}

\begin{document}

\title{An entropic partial order on a parabolic quotient of $\Ssix$}

\author{Gary McConnell}

\affiliation{Imperial College London}

\email{g.mcconnell@imperial.ac.uk}

\date{\today}

\begin{abstract}
Let $m$ and $n$ be any integers with $n>m\geq2$. Using just the entropy function it is possible to define a partial order on the symmetric group~$S_{mn}$ modulo a subgroup isomorphic to $S_m\times S_n$. We explore this partial order in the case~$m=2$ and~$n=3$, where thanks to the outer automorphism the quotient space is actually isomorphic to a parabolic quotient of~$\Ssix$.
Furthermore we show that in this case it has a fairly simple algebraic description in terms of elements of the group ring.
\end{abstract}

\maketitle

\tableofcontents

\newpage

\subsubsection{Introduction and statement of the main theorem}

Let $\prob=(p_i)$ represent any probability vector in $\RR^6$.
This paper is concerned with a partial order $\epo$ among the~720 coordinatewise permutations of $\prob$, based on the Shannon entropy function~$H(x)=-x\log x,$
which is dependent only upon the ordering of the $p_i$ and not upon their values. It arose originally in the guise of a question in quantum information theory about \emph{classicality} versus \emph{quantumness}~\cite{me}; however the structure theory turns out to be quite general. Because its natural setting is joint quantum systems the definition requires that we stipulate `subsystems' of dimensions~2 and~3 and then take the entropy of the marginal probability vectors from $\prob$ with respect to these subsystems. This construction brings with it a natural equivalence class structure and so the partial order is in fact defined only upon~60 equivalence classes of these permutations, each of size~12. We summarise this as our main theorem, as follows. Recall that the \emph{density} of a partial order on a finite set of size~n is defined to be $r/\binom{n}{2}$, where~$r$ is the number of relations which appear in the partial order, and~$\binom{n}{m}$ denotes the binomial symbol.

\begin{theorem}\label{analalg}
Let $G=\Sym{6}$ be the symmetric group on six letters and let $J$ be any one of the six parabolic subgroups of~$G$ which are isomorphic to the dihedral group of order~12. There is a partial order~$\epo$ on the right coset space~$J\backslash G$ of density~$0.47$ whose analytical description may be given solely in terms of the Shannon entropy function~$H$. Moreover it has a concise independent algebraic description in terms of group ring elements.
\end{theorem}

The proof of this theorem, together with an in-depth analysis of the structure of~$\epo$, are essentially what constitute the remainder of the paper. We must mention here that our description of~$\epo$ is unfortunately incomplete: while we believe that there are~830 relations which constitute~$\epo$ there are nevertheless four of these relations, which we shall refer to throughout as~$\mathbf{C4}$, which we have been unable to prove or disprove analytically; although the numerical evidence for their validity is compelling. So our statements about the partial order must be read with the caveat that there is still a possibility that some or all of the~$\mathbf{C4}$ are not in fact valid relations. However the structure of the remaining~826 relations of the partial order is independent of these four.

Such a partial order may in fact be described for any function $f$ instead of $H$ provided that certain convexity conditions are met: essentially we obtain a kind of `pseudo-norm' based upon the function $f$ that we choose. A curious consequence is that we may describe a whole suite of functions apparently unconnected to entropy, whose partial orders nevertheless appear numerically to mimic $\epo$ exactly. At one level this is not very surprising, since the partial order is in some sense merely a discrete approximation to the curvature of the function concerned - hence there will be many different functions whose curvature is sufficiently similar on the appropriate region of space to give the same discrete approximation. But at another level this points to a deeper connection between certain of these functions and discrete entropies: perhaps there is an easier way to model entropy-related phenomena for low-dimensional joint systems than to attack the rather difficult entropy function itself.  The space of relatively simple functions which would appear to mimic the entropy function - in this albeit limited context - is incredibly varied. For example, the function~$f(x) = \cos(\frac{2\pi}{3}x)$ seems numerically to give exactly the same partial order as~$H(x)$, despite having a markedly different curvature function; the same is true of the function~$q(x)=(\alpha x)^3-(\alpha x)^2$ when~$\alpha=\frac{4}{9}$. Moreover any slight variation in the respective coefficients~$\frac{2\pi}{3}$ or~$\alpha$ will `break' the respective partial order. However these functions are not concave on the full interval~$(0,1)$ and so the techniques of this paper will not work on them. 

As we vary the underlying function $f$, another key question arises as to how the algebraic description needs to be modified in order to reflect the new analytical structure. Both the analytic and algebraic approaches are rich topics for further study.

The constructions here are not specific to the~6-dimensional case; however dimension~6 gives the first non-trivial partial order and sadly also the last easily-tractable one. Even for~$2\times4$ (which is the next interesting case), numerical studies indicate that the number of separate relations which need to be considered is of the order of~$10^5$, the~$3\times3$ case yields around~3 million, and for~$2\times5$ it is of the order of~20 million. Also only where the dimensions are~$2\times2$ and~$2\times3$ are we able to single out a definite permutation which is guaranteed to give the maximal classical mutual information~(CMI) no matter what the probability vector chosen~\cite{me}: in all other dimensions this grows into a larger and larger set of possibilities. However the constructions of this paper may be extended to any situation where we have joint systems of dimensions $m$ and $n$: for any sufficiently well-behaved function $f$ we obtain a binary relation between certain permutations of the probabilities of the joint system, yielding what may be viewed as a partial order upon (some quotient of) the symmetric group $\Smn$ itself. We shall always assume $2\leq m\leq n$, for if $m>n$ then the situation is identical just with the subsystems reversed; if $m=1$ then there is nothing to be said since every permutation will give the same result, as will be seen from the definitions below.

We conclude this introductory section with a word on how this partial order arose. Suppose that we have ordered the~$p_i$ so that~$p_1\geq p_2\geq p_3\geq p_4\geq p_5\geq p_6$. 
In~\cite{me} it was shown that the permutation
$$(p_1,p_4,p_5,p_6,p_3,p_2)$$
will always yield the maximal CMI out of all of the possible permutations given by~$\Ssix$. This built on work in~\cite{santerdav} and~\cite{santerdavPRL} which showed that the minimum CMI of all of these permutations was contained in a set of five possibilities, all of which do in fact occur in different examples. The results on the minima were achieved solely using considerations of majorisation among marginal probability vectors; however in order to prove maximality it was necessary to invoke a more refined entropic binary relation denoted~$\rhd$. In exploring this finer ordering we found that it did indeed give rise to a well-defined partial order which moreover had a neat description in terms of symmetric group elements. So the paper is the result of this exploration.

\subsubsection{Structure of the proof of the main theorem}

We now outline how the proof of theorem~\ref{analalg} will proceed. First of all however we need to decipher the connection with the parabolic subgroups $J$, since this barely appears elsewhere in the paper. The point is that because $\Sym{6}$ has a class of non-trivial outer automorphisms~\cite{rotman} we are able to study some phenomena via their image under any particular outer automorphism of our choosing: a trick which often makes things much clearer. Let $K$ be the dihedral group corresponding to row and column swaps which we shall define in section~\ref{qmpo}. As is easy to verify, for any $J$ as described in the theorem there exists at least one outer automorphism mapping $K$ onto $J$ and so any partial order which we may define upon $K\backslash G$ will also give an isomorphic partial order on $J\backslash G$, and vice-versa. So we define our partial order in its natural context on the coset space $K\backslash G$ and then merely translate the result into the more familiar language of parabolic subgroups in the statement of the theorem. Indeed there is no reason - other than the richness of structure which has been investigated for parabolic subgroups - for phrasing it in these terms. One could equally well describe the partial order on the quotient of~$G$ by \emph{any} dihedral subgroup of order~12, for there are two conjugacy classes of subgroups of $G$ which are dihedral of order~12 - namely the class containing~$K$ and the class containing the parabolics~$J$ - each of size~60, and they are mapped onto one another by the action of the outer automorphisms.

So the proof of theorem~\ref{analalg} will go as follows. Once we establish the basic definitions regarding entropy, classical mutual information, majorisation and the entropic binary relation~$\rhd$, we begin to examine each of them in the context where two permutations differ by right multiplication by just a single transposition: first because this is the simplest case; but secondly because it actually generates all but~5 out of~186 covering relations in the partial order. A general rule for comparing pairs of permutations differing by more than one transposition under the entropic binary relation~$\rhd$, moreover, seems to be very difficult: we are fortunate that only these five \emph{`sporadic'} relations exist which cannot be generated via some concatenation of single-transposition relations. We elaborate necessary and sufficient conditions for permutations separated by a single transposition both for majorisation and for the entropic binary relation~$\rhd$, noting the result from~\cite{me} that majorisation implies~$\rhd$ but not vice-versa. This gives a total of~165 relations arising from majorisation, and~90 relations arising solely from the binary entropy relation~$\rhd$: a grand total of~255 relations arising from single transpositions. The transitive closure of these~255 relations contains~818 relations in total.

Once this is proven we shall almost have completed our description of~$\epo$, for numerically it is easy to show that with the exception of the~12 relations which are generated when the \emph{sporadic~5} are included, any other possible pairings are precluded by counterexample. So the partial order must have between~818 and~830 relations. With the two proven in theorem~\ref{proveit} the transitive closure grows to~826 relations, leaving just the set~$\mathbf{C4}$ mentioned above. This completes the `analytic' description of~$\epo$.

It then remains to prove that~$\epo$ has a neat description in terms of the group ring $\ZZ[G]$. We give an iterative algorithm for constructing the entire web of~255 single-transposition relations referred to above starting from scratch, using simple rules which have no apparent connection to entropy. Of course we would not have `seen' this description had it not been for the analytic work which went before; however once we know what we are looking for, the entire complex of~255 relations is describable in very straightforward terms. The sporadics however must be added in to both descriptions: there seems to be no easy way of unifying their structures with the bigger picture.

\subsubsection{Acknowledgments}
First of all thank you to Terry Rudolph and the QOLS group at Imperial College for their hospitality. I would also like to thank Peter Cameron, Ian Grojnowski and David Jennings for many helpful conversations.

\newpage

\section{CMI, majorisation and the entropic binary relation $\rhd$}

\subsection{The classical mutual information attached to an $m\times n$ probability matrix}

Let $N$ be any positive integer and define the usual probability simplex to be
$$\Delta^N=\{(p_1,p_2,\ldots,p_{N+1})\in\RR^{N+1}\ :\ \sum_{i=1}^{N+1} p_i=1\hbox{ and }p_i\geq0\hbox{ for all }i\}.$$
Now consider the case where $N+1=mn$ is a composite number and 
let $\prob=(p_i)\in\Delta^{mn-1}$ be any probability vector: we view this as a set of joint probabilities for two systems of size $m$ and $n$. 
We reflect the split into subsystems by arranging the $p_k$ into an $m\times n$-matrix $P$ as follows:
\begin{equation}\label{rowsncols}
\bordermatrix{
&c_1 & c_2 & \ldots & c_n\cr
r_1&p_1 & p_2 & \ldots & p_n \cr
r_2&p_{n+1} & p_{n+2} & \ldots & p_{2n} \cr
\vdots&\vdots & \vdots & \ddots & \vdots \cr
r_m&p_{(m-1)n+1} & p_{(m-1)n+2} & \ldots & p_{mn} \cr
}=P.
\end{equation}
As depicted above we let the row sums (which are the marginal probabilities for the first subsystem) be denoted by $r_i=\sum_{j=1}^n p_{(i-1)n+j}$ for $i=1,\ldots,m$ and similarly for the column sums (which are the marginal probabilities for the second subsystem): $c_j=\sum_{i=1}^m p_{(i-1)n+j}$ for $j=1,\ldots,n$.
Then given any permutation $\sigma$ in the symmetric group $\Smn$ on $mn$ letters sending $p_i$ to $p_{\sigma(i)}$
we define a new $m\times n$-matrix $P^\sigma$ as follows:
\begin{equation*}
P^\sigma=\begin{pmatrix}
p_{\sigma(1)} & p_{\sigma(2)} & \ldots & p_{\sigma(n)} \cr
p_{\sigma(n+1)} & p_{\sigma(n+2)} & \ldots & p_{\sigma(2n)} \cr
\vdots & \vdots & \ddots & \vdots \cr
p_{\sigma((m-1)n+1)} & p_{\sigma((m-1)n+2)} & \ldots & p_{\sigma(mn)} \cr
\end{pmatrix},
\end{equation*}
defining the appropriate marginal probabilities in a similar fashion.

To define the classical mutual information~\cite[\S2.3]{cover} we take the sum of the entropies of the~$r_i$ and the~$c_j$ over all~$i,j$ and then subtract the sum of the individual entropies of the~$p_k$, for~$k=1,\ldots,mn$. 

\begin{definition}\label{cmile}
With notation as above, the classical mutual information $I(P)$ of the matrix $P$ is given by
\begin{equation}\label{CMIdef}
I(P) = \sum_{i=1}^m-r_i\log r_i + \sum_{j=1}^n-c_j\log c_j - \sum_{k=1}^{mn}-p_k\log p_k.
\end{equation}
We will often write $H(x)=-x\log x$ for $x\in [0,1]$ and so we may rewrite~(\ref{CMIdef}) as
\begin{equation*}
I(P) = \sum_{i=1}^m H(r_i) + \sum_{j=1}^n H(c_j) - \sum_{k=1}^{mn}H(p_k).
\end{equation*}
\end{definition}

\subsection{Majorisation between two elements of $\Smn$}\label{majdef}

For definitions and basic results connected with majorisation, see~\cite{bhatia} and~\cite{marshall}. We shall use the standard symbol~$\succ$ to denote majorisation between vectors.
For any~$m\times n$ probability matrix~$M$, let us denote by~$\mathbf{r}(M)\in\RR^m$ the vector of marginal probabilities represented by the sums of the rows of~$M$ and similarly by~$\mathbf{c}(M)\in\RR^n$ the vector of marginal probabilities created from the sums of the columns of~$M$. Throughout the paper we shall use the symbol $H$ interchangeably for the function of one variable as well as the function on probability vectors, where if $\mathbf{v}=(v_i)\in\RR^N$ is any such vector then
$$H(\mathbf{v})=\sum_{i=1}^N H(v_i)=-\sum_{i=1}^N v_i\log v_i.$$

\begin{lemma}\label{majoris}
Let $M_1,M_2$ be two probability matrices.
If $\mathbf{r}(M_1)\succ\mathbf{r}(M_2)$ and if $\mathbf{c}(M_1)\succ\mathbf{c}(M_2)$, then
$$I(M_1)\leq I(M_2).$$
\end{lemma}

\begin{proof}
See~\cite{santerdav}: it follows from the fact that $H$ is a Schur-concave function~\cite[\S II.3]{bhatia}.
\end{proof}

It should be pointed out that the converse is definitely NOT true: indeed it is this very failure which gives substance to the definition of the entropic binary relation $\rhd$.

\begin{definition}\label{critta}
If the hypotheses of Lemma \ref{majoris} hold then we write
$$M_1\succ M_2$$
and we shall say that $M_1$ majorises $M_2$: \bf but this matrix terminology is not standard.
\end{definition}

Note that definition~\ref{critta} has nothing intrinsically to do with entropy: it is the fact that entropy is a Schur-concave function which enables us to link it to majorisation~\cite{marshall}.
By the symmetry of the entropy function $H$ upon vectors, the relation of majorisation between matrices which we have just defined is invariant under row swaps and column swaps; moreover if $m=n$ then it is also invariant under transposition. 

From now on we shall use~$\prob$ to denote a probability vector of length~$mn$ (where~$m,n$ will be clear from the context) \emph{written in non-increasing order} and~$P$ to denote the corresponding~$m\times n$ matrix derived from~$\prob$ as above by successively writing its entries along the rows. Similarly~$\prob^\sigma,P^\sigma$ will denote the respective images under an element~$\sigma\in\Smn$. Notice that our $\prob$ are thereby chosen from a much smaller convex set than $\Delta^{mn-1}$, namely from the analogue of the `positive orthant' of a vector space:
\begin{equation}
\fd_{mn} = \{\ \prob\in\Delta^{mn-1}\ :\ p_1\geq p_2\geq\ldots\geq p_{mn}\ \},
\end{equation}
which is the topological closure of a fundamental domain for the action of~$\Smn$ upon~$\Delta^{mn-1}$. Henceforth all of the probability vectors with which we shall work will be assumed to be chosen from this set~$\fd_{mn}$; the corresponding set of matrices (constructed from each $\prob\in\fd_{mn}$ as above and therefore also with entries in non-increasing order as we go along successive rows) will be denoted $\fdM_{mn}$.

\begin{definition}\label{crittasymbols}
Let $\sigma,\sigma'\in\Smn$. If $\mathbf{r}(P^\sigma)\succ\mathbf{r}(P^{\sigma'})$ and if $\mathbf{c}(P^\sigma)\succ\mathbf{c}(P^{\sigma'})$ for all $P\in\fdM_{mn}$
then we write
$$\sigma\succ\sigma'$$
and we shall say that $\sigma$ majorises $\sigma'$: \bf but again this terminology is not standard.\it
\end{definition}

We are now ready to define a finer relation than the one which majorisation gives upon permutations of a fixed probability vector. This relation is the key to all of the results in this paper.

\subsection{Definition of the entropic binary relation $\rhd$ between two elements of $\Smn$}

If we consider the class of $(mn)!$ matrices formed by permuting the entries in the matrix $P$ in~(\ref{rowsncols}) under the full symmetric group $\Smn$ and then look at the CMI of each of the resulting matrices, there is a rigid \it a priori \rm partial order which holds between them, and which does not vary as $P$ moves over the whole of $\fdM_{mn}$. That is to say, \bf it does not depend on the sizes of the $\{p_i\}$ but only upon their ordering\rm. In low dimensions, much of the partial order can be explained by majorisation considerations. However there is a substantial set of relations which depends on a much finer graining than majorisation gives. In dimension~6 this fine-graining will become our entropic partial order~$\epo$. 

We denote the individual relational operator by~$\rhd$ and define it as follows. 

\begin{definition}\label{bump}
Given permutations $\sigma,\sigma'\in\mathbf{S}_{mn}$ we say that 
$$\sigma\ \rhd\ \sigma'$$
if it can be shown that $I(P^{\sigma'})-I(P^\sigma)$ is non-negative for all $P\in\fdM_{mn}$.
That is to say, given an ordered matrix $P\in\fdM_{mn}$, the relation $I(P^\sigma)\leq I(P^{\sigma'})$ holds \bf irrespective of the relative sizes of the entries\it.
This is the same as saying that
$$H(\mathbf{r}({P^\sigma}))+H(\mathbf{c}({P^\sigma}))\ \leq\ H(\mathbf{r}({P^{\sigma'}}))+H(\mathbf{c}({P^{\sigma'}}))$$
for all $P\in\fdM_{mn}$.

In order to keep the notation consistent with that of majorisation, we have adopted the convention that $\sigma\rhd\sigma'$ corresponds to $I(P^\sigma) \leq I(P^{\sigma'})$ for all $P\in\fdM_{mn}$.
\end{definition}

\begin{remark}
A key observation at this stage is that the partial order is not really connected with the notion of classical mutual information~(CMI) so much as it is with entropy itself, for the term which is the sum of the entropies of the individual joint probabilities is common to all permutations of a given fixed matrix~$P$, and so as we pointed out in definition~\ref{bump} the ordering depends only upon \emph{the relative sizes of the sums of the entropies of the marginal probability vectors}. Indeed, nothing meaningful may be said within this framework about any relation between the~CMI of matrices whose (sets of) entries are distinct: the ordering is effectively concerned solely with permutations.
\end{remark}

Now majorisation implies $\rhd$, but not vice-versa: we have the following result which was proven in~\cite{me}. The notation $(\alpha,\beta)$ for the transposition will be clarified in the next section.
\begin{proposition}\label{majmcmaj}
Let $\sigma\in\Smn$ and let $\tau=(\alpha,\beta)\in\Smn$ be the transposition swapping elements $\alpha$ and $\beta$. Then
\begin{equation*}
\left(\sigma\succ\sigma\tau\right)  \Rightarrow  \left(\sigma\rhd\sigma\tau\right).
\end{equation*}
Furthermore if $\alpha$ and $\beta$ belong to the same row or column of the corresponding $m\times n$ matrix, then the two notions are the same.\qed
\end{proposition}

We now explore the relations which arise from single transpositions.

\subsection{The entropic binary relation $\rhd$ for a single transposition}

In order to see what the entropic binary relation is in the case which will most interest us - that of a single transposition - we once again consider a general $m\times n$ probability matrix $P=(p_i)\in\fdM_{mn}$ as depicted in~\ref{rowsncols}. 
Let $\sigma$ be some element of $G$, so our starting matrix will be $P^\sigma$.
Let $\tau$ be any transposition acting on $P^\sigma$, interchanging two elements which we shall refer to as $\alpha$ and $\beta$ (by a slight abuse of notation, since the positions and the values will be referred to by the same symbols). The following diagram illustrates this action of $\tau$ on $P^\sigma$: we write $P^{\tau\sigma}=(P^\sigma)^\tau$ for the image of $P^\sigma$ under $\tau$ since we always write abstract group actions on the \emph{left}; but note that when it comes to the comparison we are trying to effect between group elements then since $\tau$ actually multiplies $\sigma$ on the \emph{right}, we will be comparing $\sigma$ with $\sigma\tau$ as required.

\begin{equation}
\label{matricks}
\bordermatrix{
&&&&c_\beta&&c_\alpha&\cr
&p_{\sigma(1)} & p_{\sigma(2)} & \ldots & \ldots & \ldots & \ldots & p_{\sigma(n)} \cr
&p_{\sigma(n+1)} & p_{\sigma(n+2)} & \ldots & \ldots & \ldots & \ldots & p_{\sigma(2n)} \cr
&\vdots & \vdots & \vdots & \vdots & \vdots & \vdots & \vdots \cr
r_\alpha&\ldots & \ldots & \ldots & \ldots & \ldots & \alpha & \ldots \cr
&\vdots & \vdots & \vdots & \vdots & \vdots & \vdots & \vdots \cr
r_\beta&\ldots & \ldots & \ldots & \beta  & \ldots & \ldots & \ldots \cr
&\vdots & \vdots & \vdots & \vdots & \vdots & \vdots & \vdots \cr
&p_{\sigma((m-1)n+1)} & p_{\sigma((m-1)n+2)} & \ldots & \ldots & \ldots & \ldots & p_{\sigma(mn)}\cr
}=P^\sigma
\end{equation}
and under the action of $\tau$ this is mapped to:
\begin{equation}
\label{matrickstau}
\bordermatrix{
&&&&c_\beta^\tau&&c_\alpha^\tau&\cr
&p_{\sigma(1)} & p_{\sigma(2)} & \ldots & \ldots & \ldots & \ldots & p_{\sigma(n)} \cr
&p_{\sigma(n+1)} & p_{\sigma(n+2)} & \ldots & \ldots & \ldots & \ldots & p_{\sigma(2n)} \cr
&\vdots & \vdots & \vdots & \vdots & \vdots & \vdots & \vdots \cr
r_\alpha^\tau&\ldots & \ldots & \ldots & \ldots & \ldots & \beta  & \ldots \cr
&\vdots & \vdots & \vdots & \vdots & \vdots & \vdots & \vdots \cr
r_\beta^\tau&\ldots & \ldots & \ldots & \alpha & \ldots & \ldots & \ldots \cr
&\vdots & \vdots & \vdots & \vdots & \vdots & \vdots & \vdots \cr
&p_{\sigma((m-1)n+1)} & p_{\sigma((m-1)n+2)} & \ldots & \ldots & \ldots & \ldots & p_{\sigma(mn)}\cr
}=P^{\tau\sigma}.
\end{equation}

Without loss of generality we may stipulate that \bf as matrix entries \rm $\alpha>\beta$ (if they are equal there is nothing to be done). We wish to compare $I(P^\sigma)$ with $I(P^{\tau\sigma})$. Note firstly that by the definition of CMI, the difference $I(P^{\tau\sigma})-I(P^\sigma)$ depends only on the rows and columns containing $\alpha,\beta$: all of the rest of the terms vanish as they are not affected by the action of $\tau$. We denote by $r_\alpha$ (respectively $r_\beta$) the sum of the entries in the row of $P^\sigma$ which contains $\alpha$ (resp. $\beta$), and by $c_\alpha$ (respectively, $c_\beta$) the sum of the entries in the column of $P^\sigma$ which contains $\alpha$ (resp. $\beta$). Similarly, we denote by $r_\alpha^\tau,r_\beta^\tau,c_\alpha^\tau,c_\beta^\tau$ the image of these quantities under the action of $\tau$. See the diagrams~(\ref{matricks}),~(\ref{matrickstau}) above.

NB: $r_\alpha^\tau,c_\alpha^\tau$ (respectively, $r_\beta^\tau,c_\beta^\tau$) no longer contain $\alpha$ (respectively~$\beta$), but rather~$\beta$ (respectively~$\alpha$).

So the quantity we are interested in becomes
\begin{eqnarray}
I(P^{\tau\sigma})-I(P^\sigma) & = & 
H(\mathbf{r}(P^{\tau\sigma}))+H(\mathbf{c}(P^{\tau\sigma}))-H(\mathbf{r}(P^\sigma))-H(\mathbf{c}(P^\sigma)) \nonumber \\
& = & H(r_\alpha^\tau)-H(r_\alpha)+H(r_\beta^\tau)-H(r_\beta)+H(c_\alpha^\tau)-H(c_\alpha)+H(c_\beta^\tau)-H(c_\beta),\label{willit}
\end{eqnarray}
with the proviso that if $\alpha$ and $\beta$ happen to be in the same row (respectively column) then the $r_\ast^\ast$ (respectively, $c_\ast^\ast$) terms vanish.
The terms in~(\ref{willit}) are grouped in pairs of the form $\pm(H(x+(\alpha-\beta))-H(x))$, which means we may write it in a more suggestive form:
\begin{equation}\label{slump}
I(P^{\tau\sigma})-I(P^\sigma) = (\alpha-\beta)\left(-\frac{H(r_\alpha)-H(r_\alpha^\tau)}{\alpha-\beta}+\frac{H(r_\beta^\tau)-H(r_\beta)}{\alpha-\beta} -\frac{H(c_\alpha)-H(c_\alpha^\tau)}{\alpha-\beta}+\frac{H(c_\beta^\tau)-H(c_\beta)}{\alpha-\beta}\right).
\end{equation}
To take advantage of the link with calculus, we introduce Lagrangian means~\cite[VI \S2.2]{bullen}.

\begin{definition}
Let $\varphi$ be a continuously differentiable and strictly convex or strictly concave function defined on a real interval $J$, with first derivative $\varphi'$. Define the {\bf Lagrangian mean $\mu_\varphi$ associated with $\varphi$} \it to be:
\begin{equation}\label{lagrean}
\mu_\varphi(x,y)  = \begin{cases}  {\varphi'}^{-1}\left(\frac{\varphi(y)-\varphi(x)}{y-x}\right)&\mbox{if }y\neq x \\
                                    x &\mbox{if }y=x                 \end{cases}
\end{equation}
for any $x,y\in J$, where ${\varphi'}^{-1}$ denotes the unique inverse of $\varphi'$.
\end{definition}
In other words, $\mu_\varphi$ is the function which arises from the Lagrangian mean value theorem in the process of going from the points $(x,\varphi(x))$ and $(y,\varphi(y))$ subtending a secant on the curve of $\varphi$, to the unique point in $[x,y]$ where the slope of the tangent to the curve $\varphi$ is equal to that of the secant. See figure~\ref{secant}. Note that the hypothesis about strict convexity/concavity is necessary in order to ensure the uniqueness of the inverse of the derivative.

\begin{figure}[h!btp]
\begin{center}
\mbox{
\subfigure{\includegraphics[width=3in,height=3in,keepaspectratio]{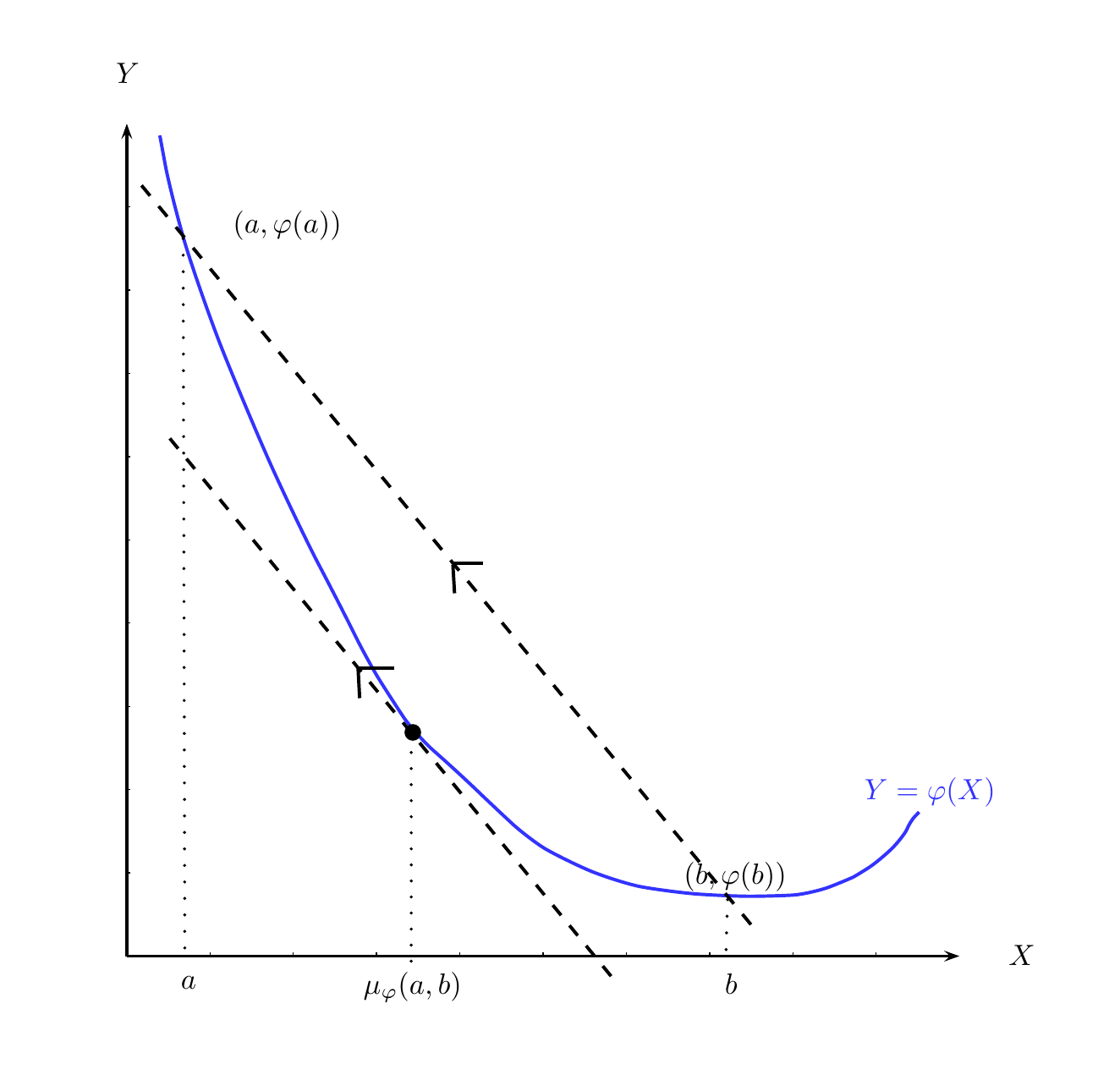}}
}
\caption{Definition of $\mu_\varphi$}\label{secant}
\end{center}
\end{figure}

If we focus on the case where $\varphi=H$ which is continuously differentiable and strictly concave on $J=[0,1]$ we may rewrite (\ref{slump}) as:
\begin{eqnarray*}
I(P^{\tau\sigma})-I(P^\sigma) & = & (\alpha-\beta)\left(-H'(\mu_H(r_\alpha^\tau,r_\alpha))+H'(\mu_H(r_\beta,r_\beta^\tau))
-H'(\mu_H(c_\alpha^\tau,c_\alpha))+H'(\mu_H(c_\beta,c_\beta^\tau))\right);
\end{eqnarray*}
indeed $\varphi'(x) = H'(x) = -(1+\log(x))$ and so this becomes:
\begin{eqnarray*}
I(P^{\tau\sigma})-I(P^\sigma) & = & (\alpha-\beta)\log \frac{\mu_H(r_\alpha^\tau,r_\alpha)\mu_H(c_\alpha^\tau,c_\alpha)}{\mu_H(r_\beta,r_\beta^\tau)\mu_H(c_\beta,c_\beta^\tau)}.
\end{eqnarray*}
Since $(\alpha-\beta)>0$, in order to determine which of the matrices gives higher CMI we only need consider the relative sizes of the numerator and denominator of the argument of the logarithm. So it is enough to study the quantity
\begin{eqnarray}
\mu_H(r_\alpha^\tau,r_\alpha)\mu_H(c_\alpha^\tau,c_\alpha) - \mu_H(r_\beta,r_\beta^\tau)\mu_H(c_\beta,c_\beta^\tau),
\label{wing}
\end{eqnarray}
as we did in~\cite{me} and as we do for the entropic binary relation below.

We are now in a position to re-state what is meant by $\rhd$ for this special case of a transposition.

\begin{lemma}
With notation as above, $\sigma\rhd \sigma\tau$ if and only if it can be shown that the quantity in (\ref{wing}) is non-negative for all $P\in\fdM_{mn}$.\qed
\end{lemma}

For convenience later on we state the following sufficient condition for $\rhd$ which is proven in~\cite{me}. Consider the four terms which constitute the first arguments of the function $\mu_H$ in (\ref{wing}), namely
\begin{equation}\label{titanic}
r_\alpha^\tau,c_\alpha^\tau,r_\beta,c_\beta.
\end{equation}
Observe that there are no \it a priori \rm relationships between the sizes of these quantities.
Let us consider the possible orderings of the four terms based upon what we know of the ordering of the matrix elements of $P$. In principle there are $24$ such possibilities; however in certain instances of small dimension such as our $2\times 3$ case, most of these may be eliminated and we are left with only a few orderings.

\begin{proposition}\label{titrate}
Suppose that the minimum element in (\ref{titanic}) is either $r_\beta$ or $c_\beta$. In addition suppose that we can verify that
$r_\beta+c_\beta \leq r_\alpha^\tau+c_\alpha^\tau$ holds for any $P\in\fdM_{mn}$.
Then $\sigma\rhd \sigma\tau$.

Conversely, suppose that the minimum element in (\ref{titanic}) is either $r_\alpha^\tau$ or $c_\alpha^\tau$ and in addition suppose that we can verify that
$r_\beta+c_\beta \geq r_\alpha^\tau+c_\alpha^\tau$ holds for any $P\in\fdM_{mn}$.
Then $\sigma\tau\rhd\sigma$.\qed
\end{proposition}

\subsection{Properties of the \emph{identric mean} $\mu_H$}

We now prove some facts specifically about $\mu_H$ which will give us an insight into the sign of the quantity in~(\ref{wing}).

\begin{lemma}\label{lollipop}
Fix $t\in(0,1)$. For $x\in(0,1-t)$:
\vspace{2 mm}

(i) $\mu_H(x,x+t)>0$ and is strictly monotonically increasing in $x$;
\vspace{2 mm}

(ii) $\mu_H(x,x+t)$ is strictly concave in $x$;
\vspace{2 mm}

(iii) $\frac{1}{e} < \frac{1}{t}(\mu_H(x,x+t)-x) < \frac{1}{2}$;
\vspace{2 mm}

(iv) $\frac{\partial}{\partial x}\left(\mu_H(x,x+t)\right)$ is strictly monotonically increasing in $t$ for fixed $x$.
\vspace{4 mm}

Let $\delta\in(0,1-t-x)$.
\vspace{2 mm}

(v) $\frac{\mu_H(x+\delta,x+\delta+t)}{\mu_H(x,x+t)}$ is monotonic decreasing in $t$ for fixed $x$;
\vspace{2 mm}

(vi) $\frac{\mu_H(x+\delta,x+\delta+t)}{\mu_H(x,x+t)}$ is monotonic decreasing in $x$ for fixed $t$.
\vspace{4 mm}

Now let $0<p<q<r<s$, with $t$ as above.
\vspace{2 mm}

(vii) Suppose that $\frac{\mu_H(q,q+t)\mu_H(r,r+t)}{\mu_H(p,p+t)\mu_H(s,s+t)}>1.$  Then $qr>ps$.

\end{lemma}

Let $y>x>0$: then we note that (iii) says that the Lagrangian mean of $x$ and $y$ occurs between $x+\frac{y-x}{e}$ and $x+\frac{y-x}{2}$. Both extremes occur in the limit, so \it a priori \rm we cannot narrow the range down further than this.

\begin{proof}
First, solving~(\ref{lagrean}) explicitly for $\varphi=H$ we see that $\mu_H$ is in fact what is known as the {\it identric mean of $x$ and $y$\rm}:
\begin{equation*}
\mu_H(x,y) = e^{-1}\left(\frac{y^y}{x^x}\right)^\frac{1}{y-x},
\end{equation*}
or if we set $t=y-x$:
\begin{eqnarray*}
\mu_H(x,x+t) & = & e^{-1}\left(\frac{(x+t)^{(x+t)}}{x^x}\right)^\frac{1}{t}\\
             & = & e^{-1}(x+t)(1+\frac{t}{x})^\frac{x}{t}\ .
\end{eqnarray*}

Now parts (i) and (ii) are proven as lemma~6 of~\cite{me} and since (iii) follows by similar techniques we omit the proof. Part~(iv) follows by taking the derivative 
$\frac{\partial^2}{\partial t\partial x}\left(\mu_H(x,x+t)\right)$ and observing that its sign is the same as the sign of
$$\frac{t^2}{x(x+t)}-\log^2(1+\frac{t}{x}),$$
which is shown to be positive in the course of proving the above lemma in~\cite{me}.

Part (vi) follows directly by taking the partial derivative of~$\frac{\mu_H(x+\delta,x+\delta+t)}{\mu_H(x,x+t)}$ with respect to $x$; taking the partial derivative with respect to $t$ instead we see that part~(v) boils down to the inequality
$$(1+\frac{t}{x})^x < (1+\frac{t}{x+\delta})^{x+\delta},$$
which on taking derivatives is seen to be a standard fact
$$\log(1+\frac{t}{x})>\frac{t}{x+t}$$
about logarithms~\cite{hardy}.

To prove (vii): let $w\leq x\leq y\leq z$ be any 4 positive real numbers arranged in the order shown, and let $\lambda\geq0$. Define a positive real function
$$\chi_\lambda(w,x,y,z) = \frac{(w+\lambda)^{(w+\lambda)}(z+\lambda)^{(z+\lambda)}}{(x+\lambda)^{(x+\lambda)}(y+\lambda)^{(y+\lambda)}}.$$
Then it follows from the explicit form for $\mu_H$ above that
$$\left[ \frac{\mu_H(q,q+t)\mu_H(r,r+t)}{\mu_H(p,p+t)\mu_H(s,s+t)}\right] ^t = \frac{\chi_0(p,q,r,s)}{\chi_t(p,q,r,s)}.$$
From now on we shall simply write $\chi_\lambda$ for $\chi_\lambda(p,q,r,s)$, for any $\lambda\geq0$, with $0<p<q<r<s$ understood as in the statement of the lemma.
Since $t>0$ by assumption and since the term in square brackets is always positive it follows that
$$\frac{\mu_H(q,q+t)\mu_H(r,r+t)}{\mu_H(p,p+t)\mu_H(s,s+t)}>1 \Leftrightarrow \frac{\chi_0}{\chi_t}>1.$$
So to prove (vii) it is enough to show that
$$\frac{\chi_0}{\chi_t}>1\ \Rightarrow\ qr>ps.$$
Now
\begin{equation}\label{ching}
\chi_\lambda' = \frac{\partial}{\partial\lambda}\left(\chi_\lambda\right) = \chi_\lambda
\log\frac{(p+\lambda)(s+\lambda)}{(q+\lambda)(r+\lambda)}
\end{equation}
and so for $\lambda\geq0$ (again noting that $\chi_\lambda$ is always positive) it follows that the sign of $\chi_\lambda'$ is exactly the sign of
\begin{equation}\label{bagosh}
(p+\lambda)(s+\lambda)-(q+\lambda)(r+\lambda)=ps-qr+((p+s)-(q+r))\lambda.
\end{equation}
So suppose $qr<ps$. Lemma~4 of~\cite{me} shows that $q+r>p+s\ \Rightarrow\ qr>ps,$ whence
$$qr<ps\ \Rightarrow\ q+r<p+s,$$
which shows that the sign in (\ref{bagosh}) must be positive for all $\lambda\geq0$. So $\chi_\lambda$ is an increasing function of $\lambda\geq0$, which means in particular that $\frac{\chi_0}{\chi_t}<1$.

Hence $\frac{\chi_0}{\chi_t}>1$ must indeed imply that $qr>ps$, as claimed.
\end{proof}

\begin{remark}
The significance of condition (vii) of lemma~\ref{lollipop} is that it may be used to derive a \emph{necessary} condition for $\rhd$, which in the $2\times3$-case in combination with proposition~\ref{titrate} yields necessary and sufficient conditions for the relation $\rhd$ between two permutations related by a single transposition. See theorem~\ref{cripes} below.
\end{remark}

\newpage

\section{The analytical construction of an entropic partial order $\epo$ for the $2\times 3$ case}\label{qmpo}

\subsection{The entropic relation $\rhd$ does give rise to a partial order}

Let $m,n\in\NN$ be arbitrary. So far we have constructed an abstract framework for the study of the binary relation~$\rhd$ between elements of $\Smn$ based on the~entropy function~$H$. Moreover we have shown that it is a necessary condition for `majorisation' between matrices related by a permutation, in the sense of definition~\ref{critta}. We now prove that it does indeed give rise to a partial order on the quotient of $\Smn$ by the subgroup~$K_{mn}$ generated by the appropriate~$H$-invariant (and so also~CMI-invariant) matrix transformations.

\begin{proposition}
The binary relation $\rhd$ gives a well-defined partial order on the coset space of the symmetric group~$\Smn$ modulo its subgroup~$K=K_{mn}$ of row- and column-swaps (together with the transpose operation if~$m=n$).
\end{proposition}

\begin{proof}
From its definition we see immediately that~$\rhd$ is reflexive and transitive. It is also anti-symmetric: let~$T_K$ be any right transversal of~$K$ in~$G=\Smn$. We need to show that if there exists a pair~$\sigma,\sigma'\in T_K$ for which both~$\sigma\rhd\sigma'$ and~$\sigma'\rhd\sigma$ hold simultaneously (meaning of course that~$I(P^\sigma)=I(P^{\sigma'})$ for all~$P\in\fdM_{mn}$) then in fact~$\sigma=\sigma'$.

We proceed by a kind of induction on the number of transpositions needed to express $\sigma^{-1}\sigma'$. Suppose that a single transposition $\tau$ takes $\sigma$ to $\sigma'$:
$$\sigma^{-1}\sigma' = \tau.$$
Our hypothesis that $I(P^\sigma)=I(P^{\sigma'})$ for all~$P\in\fdM_{mn}$ means that the quantity in~(\ref{wing}) is always zero; hence in particular its derivative with respect to~$t=\alpha-\beta$ will be zero. Recall the function~$\chi_\lambda$ which we defined in order to study the effect of varying~$t$ inside the expression~(\ref{wing}): if we look at its first partial derivative with respect to~$\lambda$ we find the expression in~(\ref{ching}). Now by our hypothesis the value of~$\chi_\lambda$ is always~1 and so the expression~(\ref{ching}) reduces to~$\log\frac{(p+\lambda)(s+\lambda)}{(q+\lambda)(r+\lambda)}$ with the~$p,q,r,s$ being some appropriate ordering of the four terms in~(\ref{titanic}). Our hypothesis implies this is identically zero, which clearly is nonsense as we vary~$\lambda$ provided we do not always have equality between the sets~$\{p,s\}$ and~$\{q,r\}$, which in the general case we do not. So for the case where~$\sigma^{-1}\sigma'$ is assumed to be a single transposition we have produced a contradiction: so indeed~$\sigma=\sigma'$.

Next suppose that 
$$\sigma^{-1}\sigma' = \tau_1\tau_2,$$
a product of two distinct transpositions. Without loss of generality we may assume that~$\tau_1$ interchanges two positions which `bracket' at most one of the positions interchanged by~$\tau_2$ (in the sense that if the two positions swapped by~$\tau_1$ are occupied by the same value $x$ then that must also be true of every other position which is in-between these positions in the ordering of the entries, and hence at most one of those positions swapped by~$\tau_2$ will be forced to be occupied by the same number $x$, but not both). If this is not the case we swap~$\tau_1$ with~$\tau_2$ and the argument will go through unchanged.
So let~$P$ be such a matrix, where the two positions swapped by~$\tau_1$ are occupied by the same value~$\delta$ say, but where one or both of the positions swapped by $\tau_2$ (depending on whether there is an overlap of one of them with~$\tau_1$) are assigned one or two different values. The key thing is that the values for~$\tau_2$ be different from one another and that at least one of them be different from~$\delta$. By construction the transposition~$\tau_1$ will have no effect on the~CMI of~$P^\sigma$, so by our hypotheses the transposition~$\tau_2$ cannot change the value either. Since we have factored out by the~$K$-symmetry of the matrices, by the \emph{strict Schur-concavity} of the entropy function~\cite[\S3A]{marshall} any two distinct column sum vectors (respectively, row sum vectors) which are not permutations of one another will yield different entropies, and therefore \emph{ceteris paribus} different CMI's. Now if~$\tau_2$ were to swap two elements of the same row, then clearly the column sum vector would change but the row vector would not, giving a different~CMI; a similar argument goes for two elements of the same column. So $\tau_2$ must be a \emph{diagonal transposition} (see definition~\ref{dhv}), swapping elements which lie both in different rows and in different columns; moreover the difference between the entropy of the row vectors before and after the action by~$\tau_2$ must be exactly equal to that between the column vectors, with the opposite sign. But then we are back to the convexity argument for the case of a single transposition above.

The general case follows by the same argument, noting that we may have to reduce either to the first or the last transposition in the expression for~$\sigma^{-1}\sigma'$ depending on the `bracketing' effect mentioned above.
\end{proof}

\subsection{The existence of a unique maximum CMI configuration in the $2\times3$ case}

For the rest of the paper we specialise to the case where~$m=2$ and~$n=3$, and we shall often merely state many of the results from~\cite{me}. The sections on definitions are identical in many places to those in~\cite{me} but are reproduced here for convenience.

From now on we denote our six probabilities by~$\{a,b,c,d,e,f\}$ and assume that they satisfy~$a\geq b\geq c\geq d\geq e\geq f\geq0$ and~$a+b+c+d+e+f=1$. In the main we shall treat these as though they were \bf strict \rm inequalities in order to derive sharper results. However we shall occasionally require recourse to the possibility that one or more of the relations be an equality: see for example the proof of theorem~\ref{cripes}.

We state the main theorem from~\cite{me}:

\begin{theorem}\label{biggun}
The matrix
$$X = \left( \begin{array}{ccc}a & d & e\\
f & c & b\end{array}\right)$$
has \bf maximal \it CMI among all $720$ possible $2\times3$ arrangements of $\{a,b,c,d,e,f\}$.

\bf This is the case \emph{irrespective of the relative sizes of} $\mathbf {a,b,c,d,e,f}$.\rm\qed
\end{theorem}

\begin{remark}\label{ream}
It is worth pointing out that one may arrive at the conclusion of theorem~\ref{biggun} by a process of heuristic reasoning, as follows. Recall from definition~\ref{cmile} that the CMI consists of three components, of which the last one is identical for all matrices which are permutations of one another. So in order to understand maxima/minima we restrict our focus to the first two terms, namely the entropies of the marginal probability vectors. Now entropy is a measure of the `randomness' of the marginal probabilities: the more uniform they are the higher will be the contribution to the CMI from these row and column sum vectors. Beginning with the columns since in general they will contribute more to the overall entropy, if we look at the a priori ordering $a>b>c>d>e>f$ it is evident that the most uniform way of selecting pairs in general so as to be as close as possible to one another would be to begin at the outside and work our way in: namely the column sum vector should read $(a+f,\ b+e,\ c+d)$. Similarly for the row sums: we need to add small terms to $a$, but the position of $f$ is already taken in the same column as $a$, so that just leaves $d$ and $e$ in the top row, and $c$ and $b$ fill up the bottom row in the order dictated by the column sums. See also the final appendix of~\cite{me} where in fact we can achieve a total ordering by the same method for the simpler case of 2x2 matrices.
\end{remark}

\subsection{The canonical matrix class representatives $\Rtt$ and the identification with a quotient of $\Ssix$}\label{scub}

There are $6!=720$ possible permutations of the fixed probabilities $\{a,b,c,d,e,f\}$, giving a set of matrices in the usual way which we shall refer to throughout as~$\Mtt$. However since simple row and column swaps do not change the CMI, and since there are $12=|\mathbf{S}_3|.|\mathbf{S}_2|$ such swaps, we are reduced to only $60=720/12$ different possible values for the CMI (provided that the probabilities $\{a,b,c,d,e,f\}$ are all distinct: clearly repeated values within the elements will give rise to fewer possible CMI values). We now classify these 60 classes of matrices according to rules which will make our subsequent proofs easier, defining a fixed set of matrices which will be referred to as~$\Rtt$. 

Throughout we shall use the symbol $K=K_6$ for the subgroup of $G=\Ssix$ generated by the row- and column-swaps referred to just now. This is the same subgroup~$K$ as we shall use in section~\ref{algview}. In the usual cycle notation
\begin{equation}\label{Kdef}
K\ \ =\ \ \langle\ \ (1,2)(4,5)\ ,\ (1,3)(4,6)\ ,\ (1,4)(2,5)(3,6)\ \ \rangle\ \ <\ \ G,
\end{equation}
where we fix for the remainder of this paper the convention that cycles multiply from right to left; so for example~$(1,2)(2,3)=(1,2,3)$ and not~$(1,3,2)$ as many authors write.
It follows that given any permutation $\sigma\in G$ the action of $K$ on rows and columns is via \emph{left multiplication}, meaning our~60 CMI-equivalence classes correspond to \emph{right cosets} of~$K$ in~$G$; whereas permutations to move us from one right~$K$-coset to another act via \emph{multiplication on the right}. 

Since we may always make $a$ the top left-hand entry of any of the matrices in~$\Mtt$ by row and/or column swaps, we set a basic form for our matrices as $M = \left( \begin{array}{ccc}a & x & y\\u & v & w\end{array}\right)$, where (as sets) $\{x,y,u,v,w\}=\{b,c,d,e,f\}$.
This leaves us with only $5!=120$ possibilities which we further divide in half by requiring that $x>y$. So our final form for representative matrices will be:
\begin{equation}
M = \left( \begin{array}{ccc}a & x & y\\
u & v & w\end{array}\right),{\rm\ with\ }x>y.
\label{mates}
\end{equation}

This yields our promised~60 representatives~$\Rtt$ in the form (\ref{mates}) for the~60 possible~CMI values associated with the \bf fixed \rm set of probabilities~$\{a,b,c,d,e,f\}$. 
We shall both implicitly and explicitly identify this set~$\Rtt$ with a set of coset representatives for~$K\backslash G$: and given two matrix classes~$M$,~$N$ the statement that~$M\succ N$ or~$M\rhd N$ will be taken to mean that the corresponding coset representatives satisfy such a relation.
We now need to subdivide~$\Rtt$ as follows. Matrices whose rows and columns are arranged in descending order will be said to be in \it standard form\rm.  It is straightforward to see that only five of the~60 matrices we have just constructed have this form, namely matrix classes~$\mathbf{1},\mathbf{7},\mathbf{13},\mathbf{25}$ and~$\mathbf{31}$ from appendix~\ref{frog} which are explicitly:
\begin{equation}
\label{minz}
\left( \begin{array}{ccc}a & b & c \\d & e & f\end{array}\right),
\left( \begin{array}{ccc}a & b & d \\c & e & f\end{array}\right),
\left( \begin{array}{ccc}a & b & e \\c & d & f\end{array}\right),
\left( \begin{array}{ccc}a & c & d \\b & e & f\end{array}\right),\textrm{ and }
\left( \begin{array}{ccc}a & c & e \\b & d & f\end{array}\right).
\end{equation}
Notice that all of these are in the form~(\ref{mates}) with the additional condition that~$u>v>w$.
If we allow the bottom row of any of these to be permuted we obtain~$5=|\mathbf{S}_3|-1$ new matrices which are not in standard form. In all this gives a total of~$30$ matrices split into five groups of~$6$, indexed by each matrix in~(\ref{minz}).

Now consider matrices in~$\Rtt$ which cannot be in standard form by virtue of having top row entries which are `too small' but nevertheless which still have the rows in descending order, viz:
\begin{equation}
\label{minzoneup}
\left( \begin{array}{ccc}a & b & f \\c & d & e\end{array}\right),
\left( \begin{array}{ccc}a & c & f \\b & d & e\end{array}\right),
\left( \begin{array}{ccc}a & d & e \\b & c & f\end{array}\right),
\left( \begin{array}{ccc}a & d & f \\b & c & e\end{array}\right),\textrm{ and }
\left( \begin{array}{ccc}a & e & f \\b & c & d\end{array}\right).
\end{equation}
Once again, by permuting the bottom row of each we obtain five new matrices: again a total of $30$ matrices split into five groups of $6$, indexed by each matrix in (\ref{minzoneup}).
This completes our basic categorization of the subsets of matrices in~$\Rtt$. 

Here are a few results from~\cite{me} which help us to classify the relations between the~$\Rtt$ classes. Call two matrices $M = \left( \begin{array}{ccc}a & p & q\\r & s & t\end{array}\right),\ \ N = \left( \begin{array}{ccc}a & x & y\\u & v & w\end{array}\right)$ \it lexicographically ordered \rm if the pair of row vectors $(a,p,q,r,s,t)$ and $(a,x,y,u,v,w)$ is so ordered (ie the word ``apqrst'' would precede the word ``axyuvw'' in an English dictionary).

\begin{lemma}
We may order the matrices in $\Rtt$ lexicographically, and majorisation respects that ordering.\qed
\end{lemma}

That is to say, if $M$ lies above $N$ lexicographically then $N$ \bf cannot \rm majorise $M$. Note that this is \bf not \rm the case for the relation $\rhd$. 

\begin{remark}
We have set out this ordering explicitly in appendix~\ref{frog}. We shall sometimes refer to matrix classes in~$\Rtt$ by these numbers: when we do so, they will appear in \bf bold \it figures as per the appendix. Equally we may refer to them by a right coset from $K\backslash G$, a representative of each of which is also tabulated in appendix~\ref{frog}.
\end{remark}

\begin{lemma}\label{treecritter}
Fix any matrix $M=\left( \begin{array}{ccc}a & x & y\\u & v & w\end{array}\right) \in\Rtt$ with the additional requirement that $u>v>w$. Permuting the elements of the bottom row under the action of the symmetric group $\Sym{3}$ we have the following majorisation relations:
\begin{equation}
\begin{array}{ccccccc}\label{flea}
\left( \begin{array}{ccc}a & x & y\\u & v & w\end{array}\right)                                 & \succ &
\begin{array}{c}\left( \begin{array}{ccc}a & x & y\\v & u & w\end{array}\right)\\
\left( \begin{array}{ccc}a & x & y\\u & w & v\end{array}\right) \end{array}                     & \succ &
\begin{array}{c}\left( \begin{array}{ccc}a & x & y\\w & u & v\end{array}\right)\\
\left( \begin{array}{ccc}a & x & y\\v & w & u\end{array}\right) \end{array}                     & \succ &
\left( \begin{array}{ccc}a & x & y\\w & v & u\end{array}\right)
\end{array}\end{equation}
There are no \rm a priori \it majorisation relations within the two vertical pairs, with the exception of the instance~$\mathbf{46}\succ\mathbf{47}$ in corollary~\ref{cannotprove}. \qed
\end{lemma}

Note that the rightmost matrix in~(\ref{flea}) corresponds to multiplication by the permutation $\varpi=(m_{21},m_{23})$ of the matrix $M=(m_{ij})$, that is:  
$\left( \begin{array}{ccc}a & x & y\\w & v & u\end{array}\right) = \varpi\left(\left( \begin{array}{ccc}a & x & y\\u & v & w\end{array}\right)\right).$
By proposition~\ref{majmcmaj} the fact that $A$ majorises $B$ implies that $I(A)<I(B)$, so the minimal value for the CMI among the \emph{representative} matrices in $\Rtt$ must occur in a matrix of the form on the left-hand side of~(\ref{flea}); conversely the maximum must occur in a matrix of the form on the right-hand side of~(\ref{flea}).

\begin{corollary}\label{bighitter}
Fix a choice of probabilities~$\{a,b,c,d,e,f\}$ as above, and consider the matrices in~$\Rtt$ as containing these fixed values. Then:

(i) there is some $M$ in (\ref{minz}) such that the \bf minimal\it\ value for the CMI of any matrix from the set $\Rtt$ is given by~$I(M)$; and

(ii) there is some $A$ in (\ref{minzoneup}) such that the \bf maximal\it\ value for the CMI of any matrix from the set $\Rtt$ is given by~$I(\varpi(A))$.\qed
\end{corollary}

\subsubsection{Aside: the basic majorisation structure in pictures}

\begin{figure}[h!btp]
\begin{center}
\mbox{
\subfigure{\includegraphics[width=6in,height=6in,keepaspectratio]{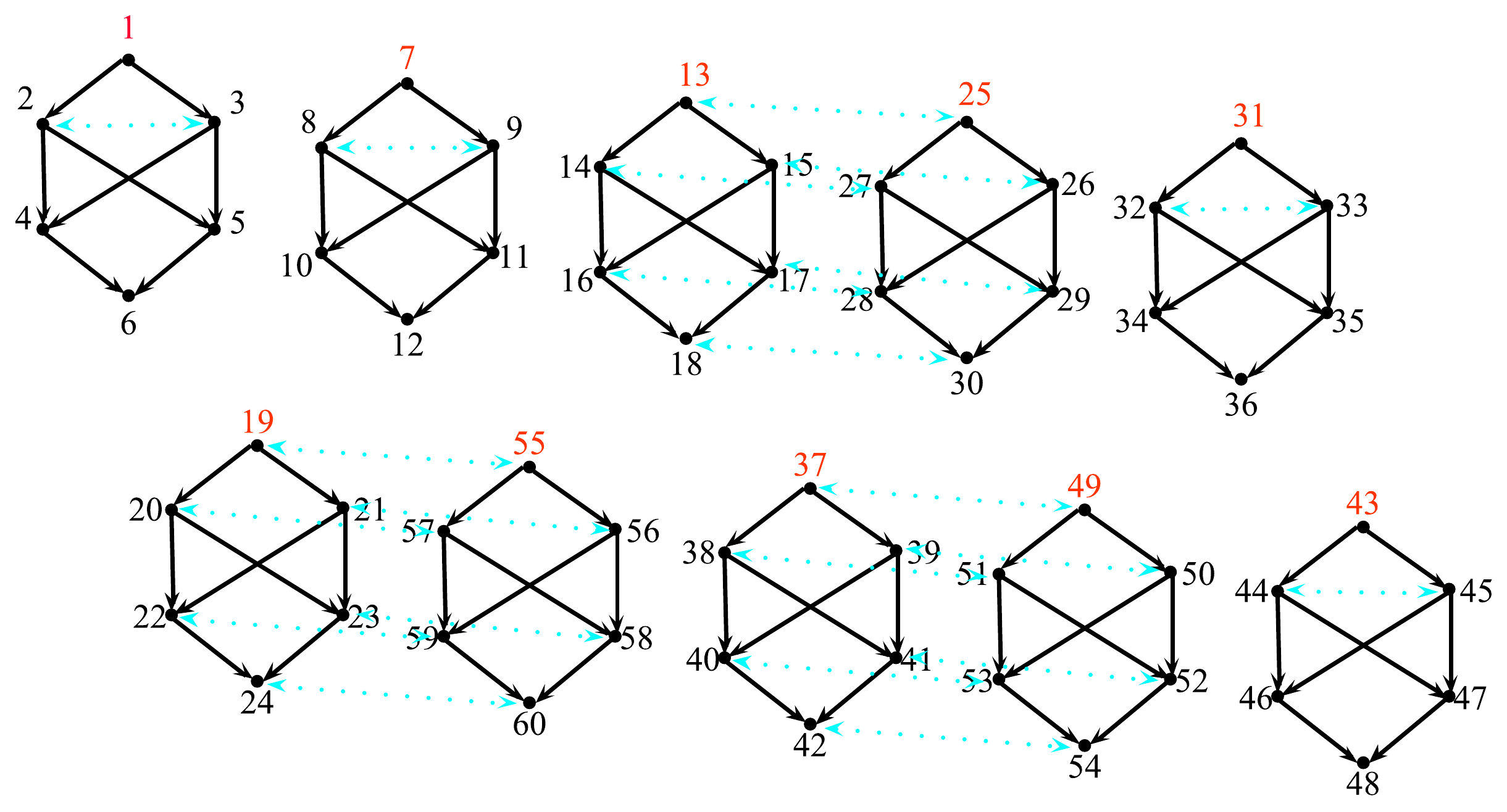}}
}
\end{center}
\caption{Representation of the most basic horizontal-transposition-based majorisation relations on $\Rtt$ together with the action of the involution $\xi$ (the blue arrows)}
\label{hexcomb}
\end{figure}

Using the simple majorisation relations developed in the foregoing discussion we have established a kind of `honeycomb' which is the backbone of the entropic partial order $\epo$ across all of $K\backslash G$.  Figure~\ref{hexcomb} shows the basic hexagonal frames corresponding to the majorisation orderings in~(\ref{flea}). The honeycomb consists of~10 hexagons each containing~6 matrices (one row of~5 slightly below the other reflecting the standard form classification), with each matrix linked via a hexagonal pattern to the other matrices in its own group.
Each hexagonal cell is in itself a diagram of the Bruhat order on $\mathbf{S}_3$. The 2 sets of 5 hexagons come from lemma~\ref{treecritter}; and the 12 lines of 5 matrices each (consisting of aligned vertices of the hexagons in their respective groupings) arise from variants of~(\ref{minz}) and~(\ref{minzoneup}). The red numbers represent the `major' element in each hexagon and are in fact all of the matrices in~(\ref{minz}) for the top row, and~(\ref{minzoneup}) for the bottom row. Note that we have placed the maximal CMI element $\mathbf{48}$ at the very bottom point, reflecting the fact that it lies below every other matrix in the $\rhd$-partial order. The minimal CMI will occur for a matrix on the very top row (matrices $\mathbf{1},\mathbf{7},\mathbf{13},\mathbf{25}$ or $\mathbf{31}$).

The numbering is as per appendix~\ref{frog}, ie the lexicographic ordering.
We have stuck to this ordering as much as possible in the diagram itself, trying to increase numbers within the hexagons as we move down and from left to right; however in places we have changed it slightly so that the patterns are rendered more clearly. The black arrows represent the majorisation relations in lemma~\ref{treecritter} which arise within each hexagon. 

The light blue double-headed arrows represent the action of the inner automorphism $\xi=\xi_\omega$ arising from the unique element $\omega=(1,6)(2,5)(3,4)\in\Ssix$ of maximal length~\cite{bjorner} which flips 22 pairs of matrix classes and fixes the remaining 16. Since this automorphism respects the binary relations $\rhd$ on $\Rtt$ it follows that any entropic binary relations (including of course majorisation) involving the nodes which have a blue arrow pointing to them will occur in pairs, thus considerably simplifying the structure. We shall explain this further in section~\ref{epoi}.

\subsection{Transpositions and the classes in $\Rtt\cong K\backslash G$}

To avoid confusion, the image of an individual matrix~$M\in\Mtt$ in~$\Rtt$ will be denoted by~$\widehat{M}$ (remember this is an equivalence class of~12 matrices and corresponds to a unique right coset of~$K$ in~$G$), and we shall denote by~$M^\ast$ its `canonical' representative in the original set~$\Rtt$ of~60 matrices: that is to say, a matrix of the form shown in~(\ref{mates}) or appendix~\ref{frog}.

We need to develop necessary and sufficient conditions for the relations~$\widehat{M}\succ \widehat{N}$ or~$\widehat{M}\rhd \widehat{N}$ in cases where \emph{matrices~$M$ and~$N$ are related by a single transposition}. However we are dealing with matrix \emph{classes}, so we need to be very clear about what we mean by saying that two matrices or matrix classes are `related by a single transposition'. 
Let $\sigma_1,\sigma_2\in G$ represent cosets $K\sigma_1,K\sigma_2\in K\backslash G$, and suppose that there is an element $\tau$ which takes $\sigma_1$ to $\sigma_2$. The translation action of $G$ is on the right, so this means that
\begin{equation}\label{tauer}
\sigma_1\tau = \sigma_2.
\end{equation}
Suppose now that we are given $k_1,k_2\in K$ and we wish to find $\tau'$ taking the representative $k_1\sigma_1$ to $k_2\sigma_2$: we find that it is 
$$\tau' = \sigma_1^{-1}k_1^{-1}k_2\sigma_1\tau,$$
and so since the middle $k_i$ product terms are all still just members of~$K$ we see that in each case we shall have a family of~12 distinct elements of $G$ mapping us between the respective cosets. So $\tau$ may well be a transposition, but its other cohorts will in general not be. However consider the case where $\tau,\tau'$ are both transpositions. Then
$$\tau'\tau=\tau'\tau^{-1}\in K^{\sigma_1^{-1}}$$
and so in particular $\tau'\tau$ must lie in the same $G$-conjugacy class as an element of $K$. But as products of two transpositions,~$K$ only contains the identity element and the elements $(1,2)(4,5),\ (1,3)(4,6)$ and $(2,3)(5,6)$, which means that either $\tau'=\tau$ or else $\tau'$ together with $\tau$ effect a column swap, viz.:
\begin{equation}\label{swop}
\sigma_1\tau'\tau^{-1}=k_1^{-1}k_2\sigma_1,
\end{equation}
which in turn implies that only this one specific `matching' transposition $\tau'$ can move us between classes where we already know there is a pair of matrices related by $\tau$.
However it is clearly NOT the case that given any element in the first class and any element in the second class, they will be related by a single transposition to one another. In summary therefore, 
when we say that two equivalence class representatives $M^\ast$ and $N^\ast$ are \emph{related by a single transposition $\tau$} we are referring to examples where there exist matrices $M,N\in\Mtt$ with $M\in\widehat{M^\ast},\ N\in\widehat{N^\ast}$ such that $M=N^\tau$. 
This means that our relations do NOT necessarily correspond to single transpositions between the class representatives of $\Rtt$.
To restrict to these would be completely artificial, depending as it does on our choice of representatives.
Hence in reading the lead-up to proposition~\ref{crikey} and theorem~\ref{cripes}, it must be borne in mind that the matrix $M$ can in principle be ANY matrix in $\Mtt$.
We set this out formally now.

\begin{definition}\label{dhv}
Let $M,N$ be any matrices in $\Mtt$ corresponding to elements $\sigma_M,\sigma_N$ respectively of $G$. If there is a transposition $\tau\in G$ such that $\sigma_M\tau = \sigma_N$ then we shall say that the matrix \emph{classes} $\widehat{M}$ and $\widehat{N}$ are \emph{related by the transposition $\tau$}.

We shall refer to a transposition as \emph{diagonal} if it swaps two elements which are neither in the same row nor in the same column as one another; \emph{vertical} if it swaps two elements of the same column: that is to say, the transposition only affects row sums; and \emph{horizontal} if it swaps two elements of the same row: in other words it only affects column sums.
\end{definition}

Let ${\cal{T}}$ be the set of pairs of distinct $\Rtt$ classes in which each representative from the first class is related to at least one representative from the second class by a single transposition.
There is a total of $|{\cal{T}}|=360$ different pairs (out of a possible $\frac{60^2-60}{2}=1770$): a result which we derive in a moment. Let $\cal{A}$ be the $60\times60$ (symmetric) adjacency matrix of the relations embodied in ${\cal{T}}$. By definition $\cal{A}$ will have 720 non-zero entries. Since transpositions generate $G$ it follows that each of the $2\times3$ matrices is eventually in some form the product of transpositions acting on a fiducial matrix (which we fix throughout to be $\left( \begin{array}{ccc}a & b & c \\d & e & f\end{array}\right)$: see appendix~\ref{frog}) and so  it is clear that the powers of $\cal{A}$ will eventually have non-zero entries everywhere, reflecting the fact that every matrix is related to every other by a finite chain of transpositions. In fact it is easy to check directly that ${\cal{A}}^3$ has no zero entries whereas ${\cal{A}}+{\cal{A}}^2$ has $720$ zeroes: hence $3$ is the maximal length of a chain of transpositions linking any two matrix classes in $\Rtt$. (For completeness we note that the equivalent figure if we were looking at all 720 matrices would be $5$ transpositions rather than $3$).
If in addition we restrict just to transpositions which fix a single point, say $a$ as we did in setting up our $\Rtt$ classes, then we need at most $4$ transpositions to navigate from any given matrix in $\Rtt$ form, to any other in that form.

It is possible to derive the number $|{\cal{T}}|$ as follows. On any of the $720$ matrices in $\Mtt$ we may act by any of $\binom{6}{2}=15$ distinct transpositions, giving a total of 10,800 relations at the level of $\Mtt$. From the explanation above it follows that each relation of the form~(\ref{tauer}) gives rise to at least 12 other single-transposition relations between the same two cosets (just multiply both sides of~(\ref{tauer}) on the left successively by elements of $K$) and so we may divide this by a factor of 12 immediately. However if two classes $\widehat{M},\widehat{N}\in\Rtt$ are related by a \it horizontal \rm transposition then in fact there will be (at least) $24$ relations between them. To see why this is so, consider the matrix $M=\left( \begin{array}{ccc}u & v & w\\x & y & z\end{array}\right)$ and without loss of generality assume that the horizontal transposition is $\tau=(1,2)$. So $M^\tau=\left( \begin{array}{ccc}v & u & w\\x & y & z\end{array}\right)$. But the class of $M^\tau$ also contains the matrix $N=\left( \begin{array}{ccc}u & v & w\\y & x & z\end{array}\right)=M^\sigma$ say, where $\sigma=(4,5)$: this is the same explanation as that of the column swaps in~(\ref{swop}) above. So because (at least) two transpositions are known to map the element $M$ of $\widehat{M}$ to an element of $\widehat{M^\tau}$, it follows from the discussion above and that regarding equation~(\ref{swop}) that there will be exactly $2\times12=24$ single-transposition relations between the classes. This behaviour cannot occur for diagonal or vertical transpositions, as is easily seen: it occurs in the $2\times n$ case for horizontal transpositions only because modulo row-swap-equivalence, both the top half and the bottom half each `tell the whole story' of the transposition. Now $6$ of the $15$ possible transpositions are horizontal (namely the right-action transpositions which would yield the same results as left action by $(1,2),(1,3),(2,3),(4,5),(4,6)$ and $(5,6)$ in each particular case), with the remaining 9 vertical or diagonal. So on $6/15$ of the relations we divide out by 24, and on the remaining $9/15$ we divide by 12. So we have $(10800*6/15)/24+(10800*9/15)/12=720$ ordered pairs. However we want unordered pairs so we divide this by 2, to obtain $|{\cal{T}}|=360$ as claimed.

In the next two sections we shall see that~255 of these~360 pairs do indeed satisfy an entropic binary relation~$\rhd$.

\subsection{Majorisation within $\Rtt$: necessary and sufficient conditions}

We now study the majorisation relations in more detail.
First we note necessary and sufficient conditions for majorisation (and hence $\rhd$, by proposition~\ref{majmcmaj}) between matrices related by a single non-diagonal transposition.
Note that we do NOT necessarily work here with matrices in the form in $\Rtt$.

Any vertical transposition $\tau=(\alpha,\beta),\ \alpha>\beta$ may be represented in the form $M = \left( \begin{array}{ccc}\alpha & x & y\\\beta & u & v\end{array}\right)$ being acted upon by switching the places of~$\alpha$ and~$\beta$. Furthermore there exists a matrix in the CMI-equivalence class of $M^\tau$ such that we may write $M^\tau = \left( \begin{array}{ccc}\alpha & u & v\\\beta & x & y\end{array}\right)$. Now since CMI is invariant in particular if we swap columns 2 and 3, it is evident (by possibly interchanging $M$ and $M^\tau$) that we may stipulate that $x>u,v,y$. So having chosen $\alpha,\beta$ our choice of $x$ is fixed. The remaining 3 letters will then have 6 possible orderings, of which exactly 4 satisfy either $u<y$ or $v<y$. Written in the above form it is evident that $M\succ M^\tau$ if and only if $x+y>u+v$ (recall that we are only interested in row sums here since the column sums are fixed, and note that $M^\tau\succ M$ is not possible since $x+y$ cannot be \it a priori \rm less than $u+v$). But $x$ is greater than both $u$ and $v$, hence \it a priori \rm majorisation will occur if and only if either $u<y$ or $v<y$. Hence for each of the $\binom{6}{2}=15$ choices of $\{\alpha,\beta\},\ \alpha>\beta$ there exist exactly 4 majorisation relations, giving a total of 60 arising from vertical transpositions.

Turning to the horizontal transpositions, they may be written in the form $M = \left( \begin{array}{ccc}\alpha & \beta & x\\y & u & v\end{array}\right)\mapsto \left( \begin{array}{ccc}\beta & \alpha & x\\ y & u & v\end{array}\right) = M^\tau$. Notice first of all that our calculations of CMI differentials will be independent of the rightmost column $\binom{x}{v}$ and so we may regard this as a majorisation comparison between vectors $(\alpha+y,\beta+u)$ and $(\alpha+u,\beta+y)$, which reduces to a contest between $u$ and $y$. We see that $M\succ M^\tau$ if and only if $y>u$ (note that this is the same relation as if we interchanged $\alpha$ with $y$ and $\beta$ with $u$ so to avoid counting twice we stipulate that $x>v$). 
Each of the $\binom{6}{2}=15$ possible pairs $\{\alpha,\beta\}$ with $\alpha>\beta$ gives us $\binom{4}{2}=6$ relations where both $x>v$ and $y>u$, yielding
a total of $90$ majorisation relations in total, arising from horizontal transpositions.

It is clear from the definitions that the respective sets of diagonal, vertical and horizontal majorisation relations are mutually exclusive. Moreover the property of being diagonal/vertical/horizontal is invariant under the equivalence relations used to construct the right cosets in $\Rtt$. Once we have theorem~\ref{cripes} below we shall have proven the following (the second part is easy to check using a program like SAGE).

\begin{proposition}\label{crikey}
There is a total of 165 distinct (strict) majorisation relations arising exclusively from transpositions between $\Rtt$ matrix classes. This comprises 15 from the diagonal transpositions, 60 from the vertical and 90 from the horizontal. By taking the transitive reduction of the directed graph on 60 nodes whose edges are the 165 majorisation relations just described, we find that 30 of them are redundant and so there are only 135 covering relations in this set. \qed
\end{proposition}

\begin{figure}
\begin{center}
\mbox{
\subfigure{\includegraphics[width=7in,height=5in]{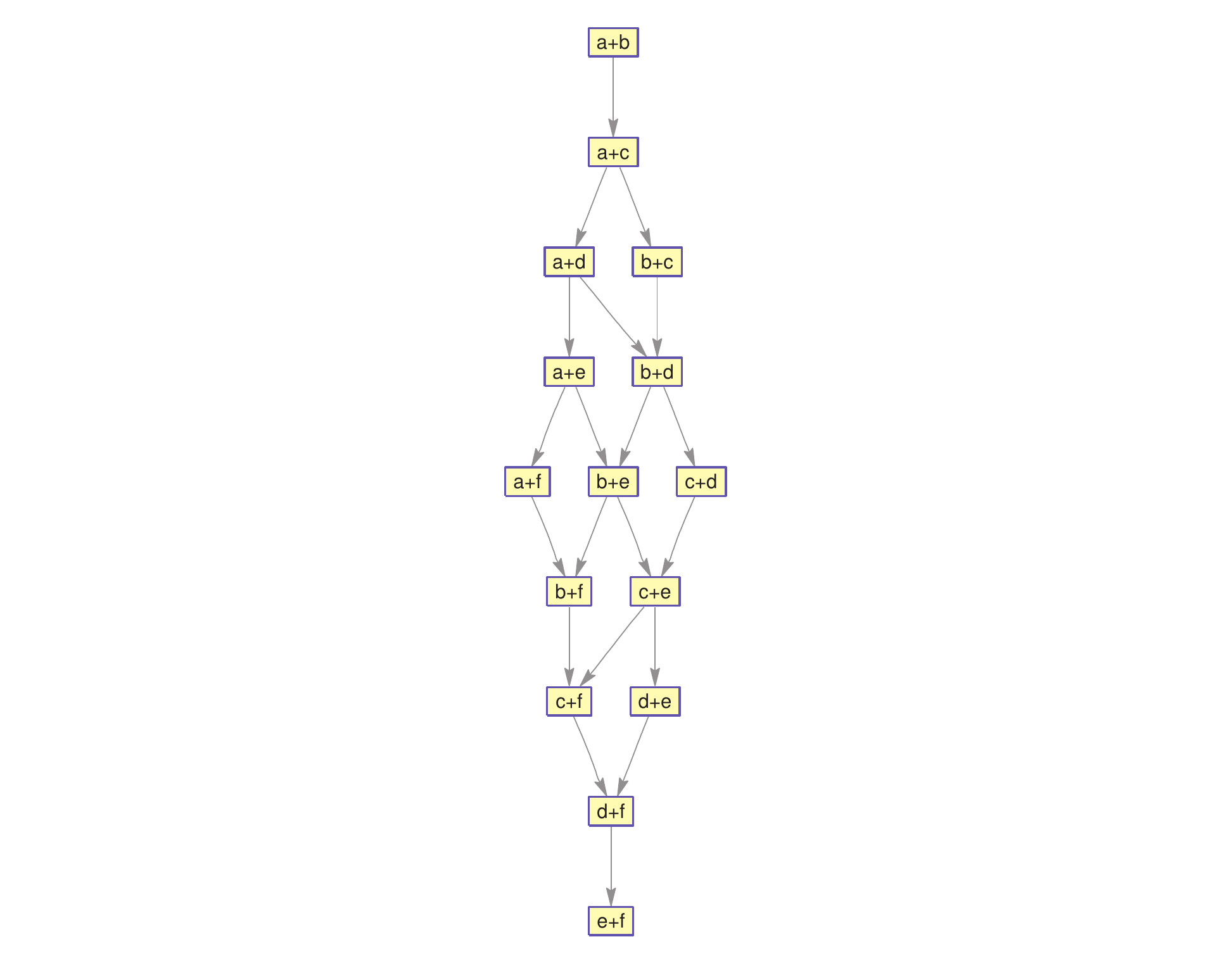}}
}
\caption{Graph GR2 of covering relations between column sum coordinates}\label{GR2}
\end{center}
\end{figure}

\begin{figure}
\begin{center}
\mbox{
\subfigure{\includegraphics[width=7in,height=5in]{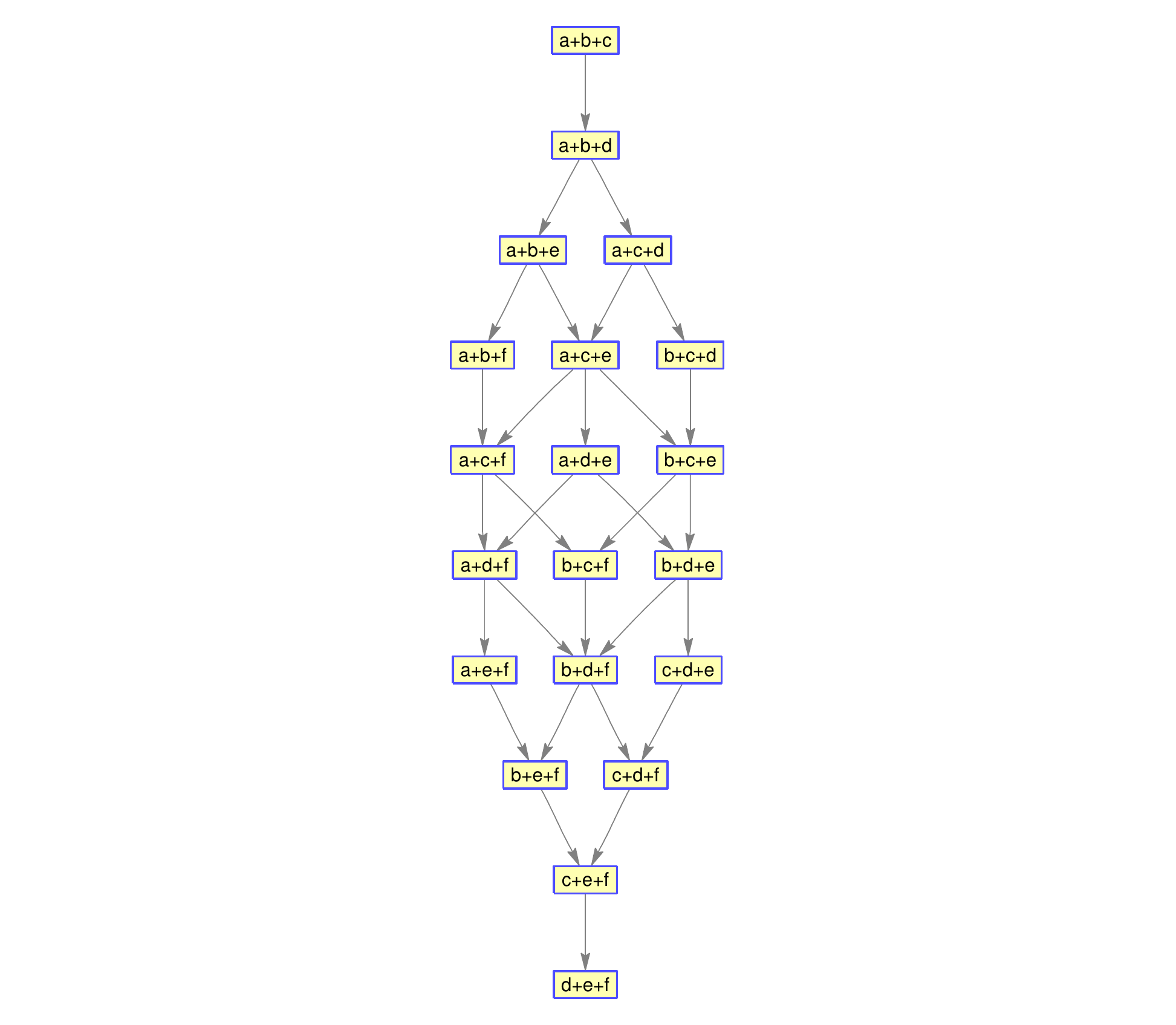}}
}
\caption{Graph GR3 of covering relations between row sum coordinates: note this is the Bruhat poset $S_6^{(3)}$ of figure 2.7 of~\cite{bjorner}}\label{GR3}
\end{center}
\end{figure}

\begin{figure}[h!btp]
\begin{center}
\mbox{
\subfigure{\includegraphics[width=4in,height=4in,keepaspectratio]{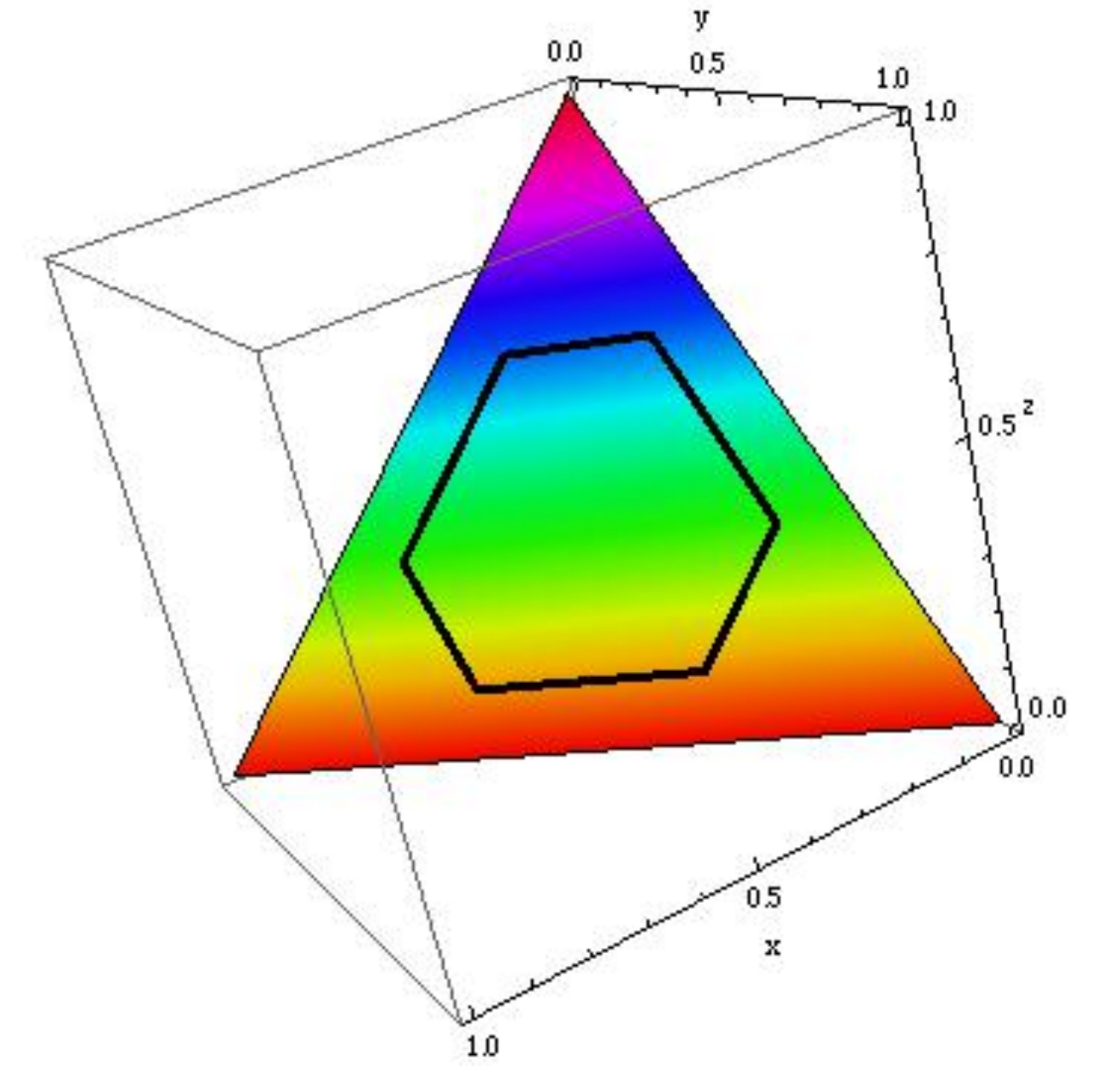}}
}
\caption{2-simplex $\Sigma_\mathbf{c}$ of possible ordered triples of values in GR2, with the hexagon $\mathbf{H}(M)$ of a matrix $M$ with $\mathbf{c}(M)=(0.6,0.3,0.1)$}
\label{herpes}
\end{center}
\end{figure}

The two figures GR2 and GR3 shown on pages~\pageref{GR2} and~\pageref{GR3} depict schematically all of the possible column sums (respectively row sums) formed from the 6 probabilities $a,b,c,d,e,f$ in the rows and columns of the matrices in $\Mtt$. It is apparent after a bit of thought that there is a 1-1 correspondence between matrices in $\Mtt$ and `compatible' pairs~$\{\mathbf{c}(M),\ \mathbf{r}(M)\}$ where~$\mathbf{c}(M)$ is a vector of~3 mutually exclusive entries from~GR2 (the column sums), and~$\mathbf{r}(M)$ is a vector of~2 mutually exclusive entries from~GR3 (the row sums), and where we mean by `compatible' that the chosen column sums \it can \rm coexist in a matrix with the chosen row sums. Moreover two matrices $M,N$ are in the same class in $\Rtt$ if and only if there are permutations $\sigma\in\mathbf{S_3},\ \tau\in\mathbf{S_2}$ such that $\mathbf{c}(M)=\mathbf{c}(N)^{\sigma}$ and $\mathbf{r}(M)=\mathbf{r}(N)^{\tau}$, where we view the actions of the groups as usual as simply permuting the coordinates of the vectors.

The arrows in~GR2 and~GR3 all indicate a `covering' relation between sums of probabilities: in other words if there is an arrow from a quantity~$X$ to a quantity~$Y$ then~$X>Y$ for \emph{every} matrix in~$\fdM_6$ and there is no quantity~$Z$ (within the possible column or row sums respectively) such that~$X>Z>Y$ for \emph{every} possible choice of matrix in~$\fdM_6$.

\begin{proposition}\label{mipelice}
Let $M,N\in\Mtt$. Then 
$M$ majorises $N$ 
if and only if:

(i) each one of the coordinates of $\mathbf{r}(N)$ lies on a (directed) path in GR3 joining some pair of coordinates of~$\mathbf{r}(M)$; \emph{and}

(ii) each one of the coordinates of $\mathbf{c}(N)$ lies on a (directed) path in GR2 joining some pair of coordinates of~$\mathbf{c}(M)$.
\end{proposition}

\begin{proof}
Let $\vec{x},\vec{y}\in\RR^n$. By a well-known result on majorisation~\cite[4.C.1]{marshall} we know that $\vec{x}\succ\vec{y}\ $~if and only if $\vec{y}$ lies in the convex hull of the $n!$ points formed by all of the permutations of the coordinates of $\vec{x}$. In our situation the column sum vectors are all elements of the 2-simplex~$\Sigma_\mathbf{c}=\{(x,y,z)\in\RR^3:\ x+y+z=1;\ x,y,z\geq0\}$, whose vertices are the units on the axes $(1,0,0),(0,1,0)$ and $(0,0,1)$. Similarly the row sum vectors lie inside a 1-simplex $\Sigma_\mathbf{r}=\{(u,v)\in\RR^2:\ u+v=1;\ u,v\geq0\}$ with endpoints~$(1,0)$ and~$(0,1)$.

Each matrix $M\in\Mtt$ gives us a column sum vector $\mathbf{c}(M)$ which in turn gives (via the permutations of its coordinates under the action of $\mathbf{S}_3$) a suite of six points whose convex hull is a closed, irregular, possibly degenerate hexagon $\mathbf{H}(M)$ lying entirely inside the closed simplex $\Sigma_\mathbf{c}$ (see figure~\ref{herpes}), whose individual coordinates are all nodes of~GR2. Similarly $M$ gives us a row sum vector $\mathbf{r}(M)$ whose convex hull is (under the action of $\mathbf{S}_2$) the line segment $\mathbf{L}(M)\subseteq\Sigma_\mathbf{r}$ and whose endpoints are $\mathbf{r}(M)$ and $\mathbf{r}(M)^\ast$, the image of $\mathbf{r}(M)$ under transposing the $u,v$ coordinates: themselves nodes of~GR3. 
Now let $N$ be any other matrix in~$\Mtt$. By definition~\ref{critta}, the hypothesis that $M\succ N$ is the same as saying that $\mathbf{r}(M)\succ\mathbf{r}(N)$ and $\mathbf{c}(M)\succ\mathbf{c}(N)$ which from above is equivalent to saying that $\mathbf{c}(N)\in\mathbf{H}(M)$ and that $\mathbf{r}(N)\in\mathbf{L}(M)$. We remark that each vector in~$\fd_6$ will give us a different hexagon and a different line segment for this same~$M$, hence the point is to show that these statements are true for every choice of vector.

So what we need to show is that $\mathbf{r}(N)\in\mathbf{L}(M)$ if and only if each one of the coordinates of $\mathbf{r}(N)$ lies on a (directed) path in GR3 joining some pair of coordinates of~$\mathbf{r}(M)$, and that $\mathbf{c}(N)\in\mathbf{H}(M)$ if and only if each one of the coordinates of $\mathbf{c}(N)$ lies on a (directed) path in GR2 joining some pair of coordinates of~$\mathbf{c}(M)$. But the arrows in~GR2 and~GR3 represent order relations between real numbers which hold for all choices of vector in~$\fd_6$. So the result follows from the definitions of~$\mathbf{L}(M)$ and~$\mathbf{H}(M)$.\end{proof}

\begin{remark}
It is clear from the foregoing that~$\left[\mathbf{H}(N)\subseteq\mathbf{H}(M){\rm\ and\ }\mathbf{L}(N)\subseteq\mathbf{L}(M)\right]$ if and only if $M\succ N$.
\end{remark}

\newpage
\subsubsection{The case of more than one transposition}

\begin{corollary}\label{cannotprove}
Let $M,N\in\Mtt$. Suppose that $M\succ N$ but that each element of $\widehat{M}$ is separated from every element of $\widehat{N}$ by a product of at least $n\geq2$ transpositions. Then with just \bf two\it\ exceptions, there is an intermediate matrix class $\widehat{L}$ separated from $\widehat{M}$ by a single transposition and from $\widehat{N}$ by $(n-1)$ transpositions, such that $M\succ L\succ N$.

The exceptions are $(34,47)$ and $(46,47)$, namely:

$$\left( \begin{array}{ccc}a & c & e\\d & f & b\end{array}\right)\succ\left( \begin{array}{ccc}a & d & e\\f & b & c\end{array}\right){\rm\ and\ }\left( \begin{array}{ccc}a & d & e\\c & f & b\end{array}\right)\succ\left( \begin{array}{ccc}a & d & e\\f & b & c\end{array}\right).$$

Both of these `exceptional' covering relations factorise once the finer relation $\rhd$ is introduced: that is to say they are no longer covering relations in $\epo$. The factorisation paths are as follows:
$$\mathbf{34}\succ\mathbf{53}\rhd\mathbf{47}{\rm\ and\ }\mathbf{46}\rhd\mathbf{34}\succ\mathbf{53}\rhd\mathbf{47}.$$
\end{corollary}

\begin{proof}
(The matrices referred to in this proof are reproduced in appendix~\ref{mxs}).

Construct (by hand, or in a simple computer program) two matrices $\mathbf{M_2}$ and $\mathbf{M_3}$ representing the transitive reductions of the partial orders in GR2 and GR3. Since GR2 has 15 nodes and 20 directed edges and GR3 has 20 nodes and 30 directed edges we obtain a $15\times15$-matrix with 20 non-zero entries for $\mathbf{M_2}$, and a $20\times20$-matrix with 30 non-zero entries for $\mathbf{M_3}$. In order to simplify things for a moment, let us speak only of column sums. Recall by proposition~\ref{mipelice} that all possible (column sum) majorisation relations for matrices $M,N$ will show up as each coordinate of $\mathbf{c}(N)=(n_1,n_2,n_3)$ lying on some directed path between two coordinates of $\mathbf{c}(M)=(m_1,m_2,m_3)$. But this is the same as saying that for each $j=1,2,3$, there exist distinct $k,l\in\{1,2,3\}$ such that some power $\mathbf{M_2}^p$ of the matrix $\mathbf{M_2}$ contains a non-zero entry at $(m_k,n_j)$ and another power $\mathbf{M_2}^q$ contains a non-zero entry at $(n_j,m_l)$. So if we form the sum (in reality a finite sum since $\mathbf{M_2}$ is nilpotent; but note that we need the identity matrix since the quantities are $\geq$ themselves):
$$\mathbf{\overline{M}_2}=\sum_{p=0}^{\infty}\mathbf{M_2}^p$$
we need only check the respective entries $(m_k,n_j)$ and $(n_j,m_l)$ of $\mathbf{\overline{M}_2}$ for $j=1,2,3$ to find whether such $k,l$ exist; if so then $\mathbf{c}(M)\succ\mathbf{c}(N)$. Similarly we form
$$\mathbf{\overline{M}_3}=\sum_{p=0}^{\infty}\mathbf{M_3}^p$$
and perform an identical procedure (with only two entries of course this time) to check for row sum majorisation. If we find non-zero entries for row and column sums in all $5=2+3$ cases then we must have $M\succ N$.

If we now look at the adjacency matrix afforded by this procedure (where we put a 1 in position (i,j) iff matrix $i$ is found to majorise matrix $j$ under this test) then we produce a $60\times60$-matrix with 423 non-zero entries. Its transitive reduction $\mathbf{T}$ has 134 non-zero entries.

If on the other hand we generate the $60\times60$ adjacency matrix of the directed graph produced by the methods of proposition~\ref{crikey} (that is to say, only using single transpositions) and take its powers we find a matrix with 421 non-zero entries, with a transitive reduction $\mathbf{T'}$ computed by SAGE to contain 135 entries.

Now if we subtract the second of these two matrices from the first we find that $\mathbf{U}=\mathbf{T}-\mathbf{T'}$ has just 5 non-zero entries as follows (recall that we use the lexicographic ordering on the $\Rtt$ matrix classes to index these adjacency matrices):
$$\mathbf{U}_{34,47}=+1;\ \mathbf{U}_{44,47}=-1;\ \mathbf{U}_{45,47}=-1;\ \mathbf{U}_{46,47}=1;\ \mathbf{U}_{46,48}=-1,$$
which is precisely what is expected if we introduce the two exceptional relations mentioned in the statement of the corollary, into the relations in $\mathbf{T'}$. (See appendix~\ref{stasi}).
\end{proof}

So all but two complicated majorisation relations will decompose into smaller majorisation relations arising from single transpositions. Indeed modulo~$\rhd$ proposition~\ref{crikey} tells the whole story of majorisation as promised in the outline of the proof of theorem~\ref{analalg}. One might hope that such a benign situation would also be the case for the relation~$\rhd$ in this~$2\times 3$ case: and indeed, there are again very few exceptions (we can prove that there are at least two, and possibly up to five). In order to establish the structure of the poset~$\epo$, we need to establish necessary and sufficient conditions for the occurrence of a relation~$\rhd$ between matrix equivalence classes which are related by a single transposition, and then as we have just done with majorisation, establish which are the exceptions.

\subsection{The entropic relation $\rhd$ in $\Rtt$: necessary and sufficient conditions}

We are able to obtain quite a dense partial ordering of the $60$ matrix classes in $\Rtt$ on the basis of the entropic partial order relation $\rhd$. Indeed almost one half of the possible pairs of distinct matrix classes are (conjecturally in the case of~4 pairs - see theorem~\ref{proveit}) related to one another: we obtain 830 relations out of a possible $\binom{60}{2}=1770$. The transitive reduction of these 830 yields 186 covering relations, as we shall show below. In the last section we found necessary and sufficient conditions for the majority of these relations which arise through `horizontal' and `vertical' transpositions and the consequent \emph{majorisation} which occurs. As per proposition~\ref{majmcmaj}, the notions of majorisation and the entropic partial order relation~$\rhd$ are the same thing in these cases: so only the `diagonal' transpositions remain to be studied.

Here we develop necessary and sufficient conditions for the relation $\rhd$ to obtain in the case of a single \emph{diagonal} transposition. Given any probability matrix in the form $M^\sigma = \left( \begin{array}{ccc}\alpha & x & y\\u & \beta & v\end{array}\right)$, for some $M\in\fdM_6$, a diagonal transposition $\tau=(\alpha,\beta),\ \alpha>\beta$ takes this to $M^{\tau\sigma} = \left( \begin{array}{ccc}\beta & x & y\\u & \alpha & v\end{array}\right)$ representing a class of CMI-invariant matrices of which one is $\left( \begin{array}{ccc}\alpha & u & v\\x & \beta & y\end{array}\right)$. Now by possibly interchanging the classes of $M^\sigma$ and $M^{\tau\sigma}$ it is clear that we may require that $x>u$. Since we are examining only binary relations between pairs of matrices we are able to require that the pairs be ordered like this for the purposes of checking whether $\sigma\rhd\sigma\tau$ or $\sigma\tau\rhd\sigma$. (Note once again that we do \bf NOT \rm assume that $x>y$ here, nor that $\alpha=a$ as we are not in general working with matrices in the form in $\Rtt$).

\begin{theorem}\label{cripes}
Let $M^\sigma = \left( \begin{array}{ccc}\alpha & x & y\\u & \beta & v\end{array}\right)$ be a matrix as above with $\alpha>\beta$ and $x>u$. For $\tau=(\alpha,\beta)$:

\rm\bf Type A:\ \ \ \ \ \it$\sigma\tau\rhd \sigma\ \Leftrightarrow\ v>y.$
\vspace{2mm}

Moreover we have a stronger relation as a sub-class of this (`type A majorisation')
$$\sigma\tau\succ \sigma\ \Leftrightarrow\ v>x>u>y\ .$$

Conversely,

\vspace{2mm}

\rm\bf Type B:\ \ \ \ \ \it
$\sigma\rhd \sigma\tau\ \Leftrightarrow\ y>x>u>v.$
\end{theorem}

\begin{proof}
It is convenient to divide the single-transposition entropic cases into two types as we have done in the statement of the theorem, which we shall henceforth refer to as type~A and type~B. Type~A is where in the above notation we are able to say that~$\mu_H(r_\alpha^\tau,r_\alpha)\mu_H(c_\alpha^\tau,c_\alpha)<\mu_H(r_\beta,r_\beta^\tau)\mu_H(c_\beta,c_\beta^\tau)$ for all matrices~$M\in\fdM_6$, which is the same as saying that~$I(M^\sigma)>I(M^{\tau\sigma})$ for all~$M\in\fdM_6$: that is, that~$\sigma\tau\rhd\sigma$. Type~B is exactly the opposite set of inequalities.

\begin{figure}[h!btp]
\begin{center}
\mbox{
\subfigure{\includegraphics[width=4in,height=4in,keepaspectratio]{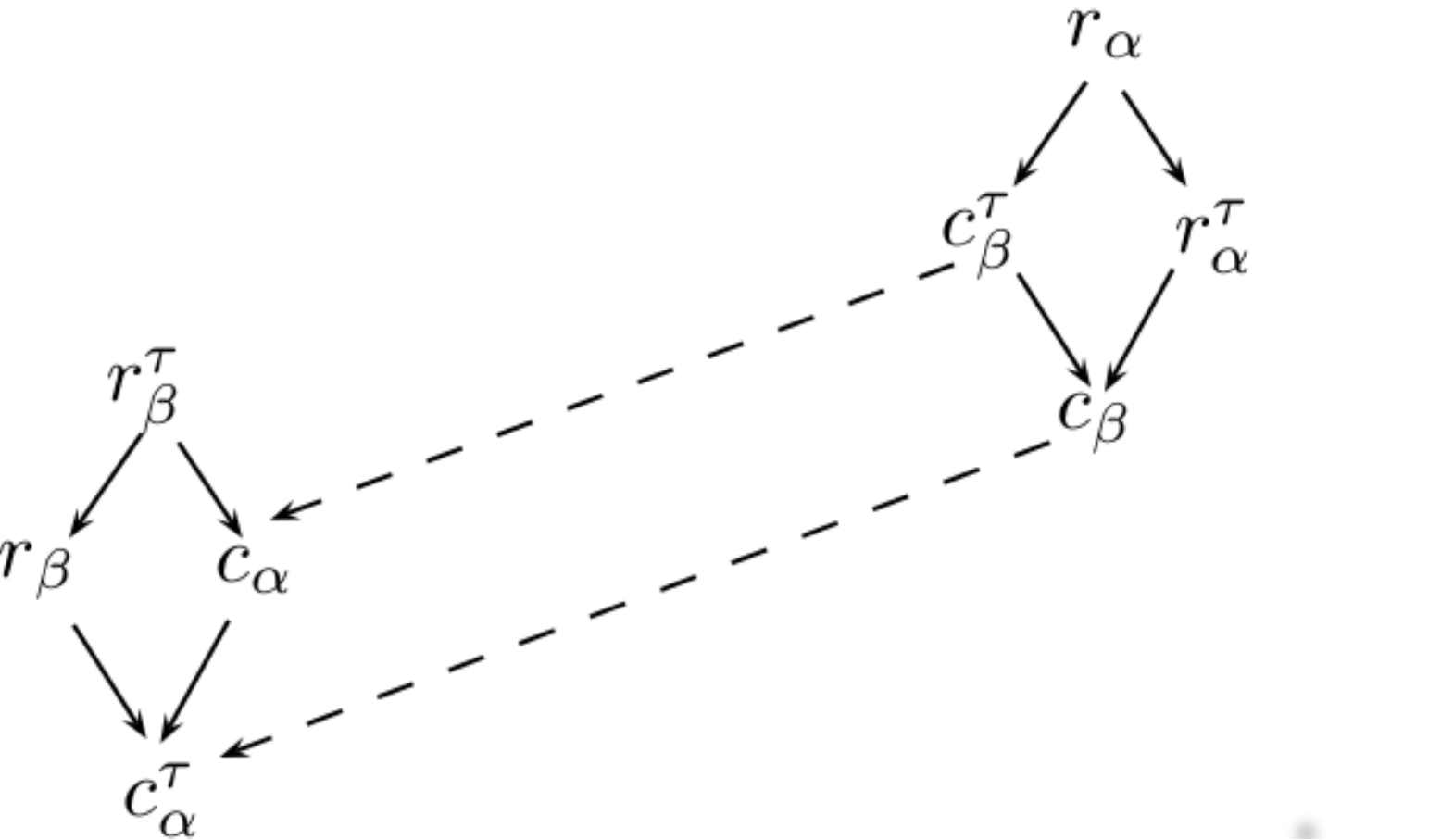}}
}
\end{center}\caption{All fixed relations between the quantities~$r_\alpha,\ r_\beta,\ c_\alpha,\ c_\beta,\ r_\alpha^\tau,\ r_\beta^\tau,\ c_\alpha^\tau,\ c_\beta^\tau$ in the $2\times n$ case assuming always that~$\alpha>\beta$. The additional dashed lines complete the picture for the case~$n=3$ where we assume in addition that~$x>u$.}\label{diamonds}
\end{figure}

Recall the quantities~$r_\alpha,\ r_\beta$ etc.~from proposition~\ref{titrate} and consider the matrix~(\ref{matricks}) for the special case where~$m=2$ (and~$n$ is any integer). We let~$\alpha,\beta$ represent any two of the~$p_{ij}$ which are in different rows and in different columns from one another. The assumption that~$\alpha>\beta$ implies the relations in figure~\ref{diamonds}, where a solid downward arrow from~$k$ to~$l$ indicates that~$k>l$.

In our case~$r_\beta=\beta+u+v,\ \ c_\beta=\beta+x,\ \ r_\alpha^\tau=\beta+x+y,\ \ c_\alpha^\tau=\beta+u$. Since the hypotheses of the theorem include the requirement that~$u<x$, we must have~$c_\alpha^\tau<c_\beta$ and it then follows from the solid lines in the diagram that~$c_\alpha^\tau$ must be the minimum of the four quantities~$r_\alpha^\tau,c_\alpha^\tau,r_\beta,c_\beta$ in~(\ref{titanic}). We have drawn in the dashed lines to reflect this additional information for the case~$n=3$ (only).
In addition it is apparent from these formulae that 
$$\left(v>y\right)\Longleftrightarrow\left(r_\beta+c_\beta>r_\alpha^\tau+c_\alpha^\tau\right),$$
so by assuming $v>y$ we shall have fulfilled the hypotheses of the second part of proposition~\ref{titrate}.
Hence the condition given for type~A is sufficient. That~$v>y$ is also a necessary condition will follow from the results on type~B which we are about to prove. We remark that since type~A and type~B are mutually exclusive, it also follows that $v<y$ is a necessary condition for type~B.

Now there are 12 possible orderings for the four values $u,v,x,y$, remembering that $u<x$ must always hold. In reverse lexicographic order these are:

\begin{enumerate}[(I)]
\item{$y>x>v>u$}
\item{$y>x>u>v$}
\item{$y>v>x>u$}
\item{$x>y>v>u$}
\item{$x>y>u>v$}
\item{$x>v>y>u$}
\item{$x>v>u>y$}
\item{$x>u>y>v$}
\item{$x>u>v>y$}
\item{$v>y>x>u$}
\item{$v>x>y>u$}
\item{$v>x>u>y$}
\end{enumerate}

We should point out here that any of the above inequalities may be relaxed to~$\geq$: of course if~$\alpha$ or~$\beta$ (or any other variable) should happen to be in-between two values of~$u,v,x,y$ which are equal then they shall also be forced to be equal to their neighbours - but this does not affect any of the arguments below. However because of this we shall need to prove strict violations of inequalities (that is, if we are trying to prove a contradiction to some expression~$f>g$ then we shall need to provide an example where actually~$f\lneq g$).


One sees straight away that cases~VI, VII, IX, X, XI and~XII are all of type~A, since~$v>y$. We now proceed to show that case~II is the only type~B and that the remaining cases (I, III, IV, V and~VIII) are neither type~A nor type~B.
We first claim that
\begin{equation}\label{finally}
r_\beta c_\beta < r_\alpha^\tau c_\alpha^\tau
\end{equation}
is a necessary and sufficient condition for type~B. Consider once again the fundamental expression~(\ref{wing}).
Recalling that~$c_\alpha^\tau$ is the smallest of the four terms,~(\ref{finally}) implies that we must have $r_\beta<r_\alpha^\tau$ and $c_\beta<r_\alpha^\tau$.
Hence setting~$p=c_\alpha^\tau$, $q=\min\{c_\beta,\ r_\beta\}$, 
$r=\max\{c_\beta,\ r_\beta\}$ and $s=r_\alpha^\tau$ gets us into the situation of the reverse implication of part~(vii) of lemma~\ref{lollipop}, namely we know~$qr<ps$: so it follows that 
$$\frac{\mu_H(r_\beta,r_\beta^\tau)\mu_H(c_\beta,c_\beta^\tau)}{\mu_H(r_\alpha^\tau,r_\alpha)\mu_H(c_\alpha^\tau,c_\alpha)}<1,$$
which by definition means type~B. So~(\ref{finally}) is a sufficient condition for type~B. We now show it is also necessary. Using the explicit formulae above for the row and column sums we see that~(\ref{finally}) is the same as the condition
\begin{equation}\label{wally}
v\beta+vx<y\beta+yu,
\end{equation}
and so we may write the \emph{reverse} inequality as:
\begin{equation}\label{endlich}
\frac{x+\beta}{u+\beta}>\frac{y}{v}\ .
\end{equation}
Since~$y>v$ is a necessary condition for type~B as observed above, we shall have proven the necessity of~(\ref{finally}) for type~B if we can prove the following:
\begin{claim}
If \rm(\ref{endlich}) \it holds then $y\leq v$.
\end{claim}

For suppose to the contrary that we have some matrix $M\in\fdM_6$ satisfying both~(\ref{endlich}) and $y>v$, so in particular we must be in one of the situations~I, II, III, IV, V or~VIII above. In probability distributions of type I, II, III and~VIII we may set $u=x$ and so $y=v$, a contradiction. In~IV and~V we may set~$y=x$ and~$u=v$ and since we can always construct an example where~$\beta>0$, we have $ux+\beta u>ux+\beta x$ which is a contradiction since $x>u$. This proves the claim.

Since (\ref{wally}) is equivalent to~(\ref{finally}), to complete the proof of the theorem for type~B it only remains to show that~(\ref{wally}) is equivalent to the condition~II, namely~$y>x>u>v$. Now~II certainly implies~(\ref{wally}), so we just need to prove that~(\ref{wally}) implies~II.
Our hypotheses include the assumption that~$x>u$ so it is enough to show that~$y>x$ and~$u>v$. 
Recall that we are still in one of the cases~I, II, III, IV, V or VIII, because $v>y$ would produce an immediate contradiction to~(\ref{wally}) since $x>u$.
Suppose that~$v\gneq u$ (ie forcing us into cases~I, III and~IV): then setting~$v=y$ we see that~(\ref{wally}) reduces to~$vx<vu$, a contradiction to~$x>u$. So~$u>v$ as required. Similarly suppose that~$x\gneq y$ (ie cases~V and~VIII): then again setting $v=y$, the inequality~(\ref{wally}) contradicts~$x>u$. So $y>x$, completing the picture that condition~II is a necessary and sufficient condition for type~B.

We now prove that the condition~$v>y$ is necessary for type~A. 
Suppose to the contrary that we have type~A but that $y>v$. By figure~\ref{diamonds} we know that $c_\beta<r_\alpha^\tau$ and by the formulae above $y>v$ implies $r_\beta<r_\alpha^\tau$, so we are again in the situation of lemma~\ref{lollipop}~(vii), with $p=c_\alpha^\tau$, $q=\min\{c_\beta,\ r_\beta\}$, 
$r=\max\{c_\beta,\ r_\beta\}$ and $s=r_\alpha^\tau$. With these definitions, type~A is synonymous with the condition
$$\frac{\mu_H(q)\mu_H(r)}{\mu_H(p)\mu_H(s)}>1,$$
and so the lemma implies that $r_\beta c_\beta>r_\alpha^\tau c_\alpha^\tau$ which we know from above is equivalent to~(\ref{endlich}). But the claim above showed that this cannot hold under the assumption that $y>v$, yielding the desired contradiction.

This completes the proof of the central assertions of the theorem. It remains to show that  if majorisation occurs for a diagonal transposition then it must be in the situation of condition~XII, and conversely that in the sub-class of type~A where~$v>x>u>y$ in fact we have majorisation.
The latter follows immediately on substituting these relations into~$M^\sigma$ and~$M^{\tau\sigma}$. 
Conversely, consider the column sums: since~$\alpha>\beta$ and~$x>u$ by hypothesis it follows that~$x+\alpha>u+\alpha>u+\beta$ and~$x+\alpha>x+\beta>u+\beta$, hence the columns of~$M^{\tau\sigma}$ must always majorise those of~$M^\sigma$. In particular this rules out `type~B majorisation'. So the only type of majorisation which is possible in this diagonal transposition setup is type~A. Suppose then that~$M^{\tau\sigma}\succ M^\sigma$. By considering the row sums this time we see that~$\alpha+u+v>\alpha+x+y$, ie~$u+v>x+y$. Since~$x>u$ we must have that~$v>x$ and~$u>y$ (to see this, consider once again the diagram~GR2 on page~\pageref{GR2}). So we may conclude that a necessary condition for type~A majorisation is that~$v>x>u>y$. So we have proven the claim about majorisation.
\end{proof}

\begin{corollary}\label{yep}
Given a probability distribution $a>b>c>d>e>f$ as above, for any ordered pair $\alpha>\beta$ chosen from $\{a,b,c,d,e,f\}$ there exist precisely $7$ diagonal entropic relations, of which exactly one is moreover a majorisation relation. Since there are $\binom{6}{2}=15$ such ordered pairs, there exist exactly $105$ diagonal entropic relations between the matrices in $\Rtt$ arising solely from transpositions. Furthermore 90 of these CANNOT be derived by majorisation considerations.
\end{corollary}

\begin{proof}
Given any one of the 15 possible pairs $(\alpha,\beta)$ with $\alpha>\beta$: exactly one of the $\frac{4!}{2}=12$ configurations of the remaining letters $u,v,x,y$ (remembering always that $u<x$) satisfies $v<u<x<y$, and 6 satisfy $v>y$, of which one further satisfies $v>x>u>y$. This means of course that 5 of the remaining configurations satisfy neither type~A nor type~B.
\end{proof}

We now deal with the situation when there is more than one transposition.

\subsection{The case of more than one transposition: the `sporadic 5'}

\subsubsection{Definition of the `sporadic 5' and proof of two of the relations}

Recall that a relation $x>y$ in a partial order is called a {\it covering relation\rm} if no $z\neq x,y$ may be found such that $x>z>y$.

\begin{theorem}\label{proveit}
There are at least~2 and at most~5 covering relations between equivalence classes in~$\Rtt$ which arise exclusively from products of two or more transpositions. They are:
\begin{eqnarray*}
\mathbf{15}\rhd\mathbf{10}:\ \ \left( \begin{array}{ccc}a & b & e\\d & c & f\end{array}\right) & \rhd &\left( \begin{array}{ccc}a & b & d\\e & f & c\end{array}\right)\ \ \hbox{(proven below)}\\
\mathbf{26}\rhd\mathbf{10}:\ \ \left( \begin{array}{ccc}a & c & d\\b & f & e\end{array}\right) & \rhd &\left( \begin{array}{ccc}a & b & d\\e & f & c\end{array}\right)\ \ \hbox{(proven below)}\\
\mathbf{37}\rhd\mathbf{11}:\ \ \left( \begin{array}{ccc}a & c & f\\b & d & e\end{array}\right) & \rhd &\left( \begin{array}{ccc}a & b & d\\f & c & e\end{array}\right)\ \ \hbox{(conjectured below)}\\
\mathbf{43}\rhd\mathbf{11}:\ \ \left( \begin{array}{ccc}a & d & e\\b & c & f\end{array}\right) & \rhd &\left( \begin{array}{ccc}a & b & d\\f & c & e\end{array}\right)\ \ \hbox{(conjectured below)}\\
\mathbf{49}\rhd\mathbf{11}:\ \ \left( \begin{array}{ccc}a & d & f\\b & c & e\end{array}\right) & \rhd &\left( \begin{array}{ccc}a & b & d\\f & c & e\end{array}\right)\ \ \hbox{(conjectured below)}.
\end{eqnarray*}
\end{theorem}

We shall require an $n$-dimensional analogue of lemma~5 of~\cite{me}.

\begin{lemma}\label{moa}
Let $\vec{v}=(v_i),\ \vec{w}=(w_i)$ be two vectors in $\RR^n$ with non-negative entries and suppose that $\vec{v}\succ\vec{w}$. Let $\phi$ be any strictly log-concave function defined on $\RR^+$. 
Then
$$\prod_{i=1}^n\phi(v_i) < \prod_{j=1}^n\phi(w_j).$$
\end{lemma}
\begin{proof}
By chapter 3, E.1 of~\cite{marshall} the product of $\phi$ on the components is strictly Schur-concave.
\end{proof}

\begin{proof}[Proof (of the theorem)]
A general rule similar to that in theorem~\ref{cripes} for cases where matrix classes are related by two transpositions seems to be very difficult to formulate. So to avoid having to do this we first of all invoke the following empirical result. We constructed a program on Matlab which easily shows by counterexample that any pairs \emph{not} related by a sequence of covering relations arising from proposition~\ref{crikey}, theorem~\ref{cripes} and/or the above list of five, will \emph{not} have any~$\rhd$ relations between them. It never seems to require more than $10^6$ randomly chosen probability vectors (just using the~\emph{rand(1,6)} function on Matlab with no modifications other than normalisation) in order to find a counterexample in any given instance - usually of course one needs far fewer than this. So it remains to prove that the two relations above indeed do hold, and we shall be done.

We remark that the first and second relations, and the third and fifth relations, are each pairs of relations which are images of one another under the automorphism~$\xi_\omega$ (see  appendix~\ref{transco}). So our proof that~$\mathbf{26}\rhd\mathbf{10}$ actually points us to a kind of `mirror image' proof of the relation~$\mathbf{15}\rhd\mathbf{10}$; and we would expect similarly for~$\mathbf{37}\rhd\mathbf{11}$ and~$\mathbf{49}\rhd\mathbf{11}$.

We first show that~$\mathbf{26}\rhd\mathbf{10}$. We have to prove that
$$H(a+e)+H(b+f)+H(c+d)+H(a+b+d)+H(c+e+f) \geq H(a+b)+H(c+f)+H(d+e)+H(a+c+d)+H(b+e+f),$$
which using the same technique as in~(\ref{slump}) we may rewrite as
\begin{equation}\label{smog}
(b-c)\log\frac{\mu_H^{(b-c)}(a+c)\ \mu_H^{(b-c)}(c+e+f)}{\mu_H^{(b-c)}(a+c+d)\ \mu_H^{(b-c)}(c+f)}+(a-d)\log\frac{\mu_H^{(a-d)}(c+d)}{\mu_H^{(a-d)}(d+e)}\geq0,\end{equation}
where we have written~$\mu_H^t(x)$ for $\mu_H(x,x+t)$. We have added in a "dummy" factor~$H(a+c)$ and then taken it out again, which has enabled us effectively to `factorise' the path from~$\mathbf{26}$ to~$\mathbf{10}$ via the matrix class~$\mathbf{8}$.
The monotonicity in~$x$ of~$\mu_H^t(x)$ for fixed~$t$ (lemma~\ref{lollipop}(i)) shows that the second term is always~$>0$ (indeed this is simply the expression which shows directly that~$\mathbf{8}\rhd\mathbf{10}$); so since~$a>b>c>d$ the left-hand side of~(\ref{smog}) is greater than or equal to
\begin{equation}\label{ta}
(b-c)\log\frac{\mu_H^{(b-c)}(a+c)\ \mu_H^{(b-c)}(c+e+f)\ \mu_H^{(a-d)}(c+d)}{\mu_H^{(b-c)}(a+c+d)\ \mu_H^{(b-c)}(c+f)\ \mu_H^{(a-d)}(d+e)},
\end{equation}
and once again by lemma~\ref{lollipop}(i) we know that $\mu_H^{(b-c)}(a+c)\geq\mu_H^{(b-c)}(a+d)$ so~(\ref{ta}) is in turn greater than or equal to the following expression:
\begin{equation}\label{simmo}
(b-c)\log\frac{\mu_H^{(b-c)}(a+d)\ \mu_H^{(b-c)}(c+e+f)\ \mu_H^{(a-d)}(c+d)}{\mu_H^{(b-c)}(a+c+d)\ \mu_H^{(b-c)}(c+f)\ \mu_H^{(a-d)}(d+e)},
\end{equation}
which has the added symmetry that the sum of the arguments of the various $\mu_H$'s in the numerator equals the sum of the arguments in the denominator. 
So we may compare these vectors of arguments and we find that
\begin{equation}\label{majo}
(a+c+d,\ c+f,\ d+e)\ \succ\ (a+d,\ c+d,\ c+e+f),
\end{equation}
since in $\RR^3$ a necessary and sufficient condition that a vector $\vec{v}$ majorise $\vec{w}$ is that $\vec{v}$ contain the overall maximum of all~6 components of $\vec{v},\vec{w}$ (in this case $a+c+d$) as well as the overall minimum (in this case either $c+f$ or $d+e$). Since $b>c$ we shall be done if we can show that the argument of the logarithm in~(\ref{simmo}) is~$\geq1$.

We now claim that
\begin{equation}\label{moo}
\frac{\mu_H^{(b-c)}(a+d)\mu_H^{(b-c)}(c+e+f)\mu_H^{(a-d)}(c+d)}{\mu_H^{(b-c)}(a+c+d)\mu_H^{(b-c)}(c+f)\mu_H^{(a-d)}(d+e)}>1.
\end{equation}

First we note that 
$$\frac{\partial^2}{\partial x^2}\left(\log(\mu_H(x,x+t))\right) = -\frac{1}{tx+x^2}$$
which proves that $\mu_H(x,x+t)$ is strictly log-concave in $x$ for fixed $t$. So if the terms $\mu_H^t(X)$ in~(\ref{moo}) all had their $t$-terms equal then~(\ref{majo}) would give us our result, by lemma~\ref{moa}. The strategy therefore is to replace the rightmost top and bottom terms in~(\ref{moo}) respectively by terms of the form $\mu_H^{(b-c)}(X+(c-e))$ and $\mu_H^{(b-c)}(X)$ whose ratio is less than or equal to $\frac{\mu_H^{(a-d)}(c+d)}{\mu_H^{(a-d)}(d+e)}$: provided that the corresponding majorisation relation still holds then we shall have finished. Note that by lemma~\ref{lollipop}(v)
$$\frac{\mu_H^{(b-c)}(c+d)}{\mu_H^{(b-c)}(d+e)}\ >\ \frac{\mu_H^{(a-d)}(c+d)}{\mu_H^{(a-d)}(d+e)},$$
while part~(vi) tells us that for any $\epsilon\in(0,1-b-d)$:
$$\frac{\mu_H^{(b-c)}(c+d)}{\mu_H^{(b-c)}(d+e)}\ >\ \frac{\mu_H^{(b-c)}(c+d+\epsilon)}{\mu_H^{(b-c)}(d+e+\epsilon)};$$
that is to say, increasing the arguments of the numerator and denominator by the same amount will \emph{decrease} the value of the expression. So we know that such an~$X$, if it exists, must be greater than~$d+e$.
However we cannot increase the arguments so as to disrupt the majorisation relation~(\ref{majo}), which means that the maximum value of the new argument in the numerator cannot be greater than~$a+c+d$, which in turn translates into the value of~$X$ being less than~$a+d+e$. (Note that the \emph{minimum} in~(\ref{majo}) will not be violated because~$c+f$ is still a component of the vector of arguments of the denominator). So again using lemma~\ref{lollipop}(vi) we see by continuity that such an~$X$ must exist provided we can prove that
$$\frac{\mu_H^{(b-c)}(a+d+e+(c-e))}{\mu_H^{(b-c)}(a+d+e)}<\frac{\mu_H^{(a-d)}(c+d)}{\mu_H^{(a-d)}(d+e)}.$$
Now the internality of the identric mean~\cite{bullen} guarantees that~$\mu_H^{(a-d)}(d+e)<\mu_H^{(b-c)}(a+d+e)$ and that~$\mu_H^{(a-d)}(c+d)<\mu_H^{(b-c)}(a+c+d)$ which together with lemma~\ref{lollipop}(i) gives us the following ordering:
$$\mu_H^{(a-d)}(d+e)<\{\mu_H^{(a-d)}(c+d),\ \mu_H^{(b-c)}(a+d+e)\}<\mu_H^{(b-c)}(a+c+d).$$
So if we can show that the sum of the central two terms exceeds that of the outer two terms then by lemma~4 of~\cite{me} we shall be done (alternatively, apply lemma~\ref{moa} to the function $\phi(x)=x$). This is equivalent to showing that
\begin{equation}\label{slopes}
\mu_H^{(a-d)}(c+d)-\mu_H^{(a-d)}(d+e)\ >\  \mu_H^{(b-c)}(a+c+d)-\mu_H^{(b-c)}(a+d+e).
\end{equation}
But the difference between the pairs of arguments on both sides is the same value $(c-e)$, so this becomes a question about the relative steepness of $\mu_H^{(b-c)}$ and $\mu_H^{(a-d)}$.
We know that $\mu_H^t(x)$ itself is strictly concave in~$x$ by lemma~\ref{lollipop}(ii), so we may define new Lagrangian means $\mathfrak{m}=\mu_{\mu_H^{(a-d)}}^{(c-e)}(d+e)$ and $\mathfrak{M}=\mu_{\mu_H^{(b-c)}}^{(c-e)}(a+d+e)$ which by the internality of the Lagrangian mean~\cite{bullen} VI.2.2 satisfy $\mathfrak{m}<\mathfrak{M}$. Denote by~$\mu_H^{t\ '}(\xi)$ the slightly more awkward expression~$\frac{\partial}{\partial x}\left(\mu_H^t(x)\right)\mid_{x=\xi}$.
Dividing~(\ref{slopes}) through by a factor of $(c-e)$ we obtain
$$\mu_H^{(a-d)\ '}(\mathfrak{m}) > \mu_H^{(b-c)\ '}(\mathfrak{M}),$$
which is what we now must prove.
But using lemma~\ref{lollipop} once again:
$$\mu_H^{(a-d)\ '}(\mathfrak{m}) > \mu_H^{(b-c)\ '}(\mathfrak{m}) > \mu_H^{(b-c)\ '}(\mathfrak{M}),$$
where the first inequality is from part~(iv) and the second from part~(ii). This completes the proof that~$\mathbf{26}\rhd\mathbf{10}$.

To prove that~$\mathbf{15}\rhd\mathbf{10}$ we need only mimic the above proof replacing each probability $a,b,c,d,e,f$ by its respective image~$f,e,d,c,b,a$ under the obvious linear extension of~$\xi_\omega$ and then reversing all the signs. With a little care, the proof goes through exactly as above; we shall just mention the key points. One word of warning: using our abbreviated notation~$\mu_H^t(x)$ for~$\mu_H(x,x+t)$ can be a little confusing because the image under~$\xi_\omega$ will be~$\mu_H^{-\xi_\omega(t)}(\xi_\omega(x+t))$.

The equivalent of~(\ref{ta}) will be:
\begin{equation}\label{tata}
(d-e)\log\frac{\mu_H^{(d-e)}(a+e)\ \mu_H^{(d-e)}(c+e+f)\ \mu_H^{(c-f)}(b+f)}{\mu_H^{(d-e)}(a+b+e)\ \mu_H^{(d-e)}(e+f)\ \mu_H^{(c-f)}(d+f)},
\end{equation}
and our corresponding move to obtain something in the form of~(\ref{simmo}), with comparable vectors of arguments on the top and the bottom, is to add~$c-d$ to the argument of~$\mu_H^{(d-e)}(e+f)$, giving us finally the following expression which we must show is always~$\geq1$:
\begin{equation}\label{somme}
\frac{\mu_H^{(d-e)}(a+e)\ \mu_H^{(d-e)}(c+e+f)\ \mu_H^{(c-f)}(b+f)}{\mu_H^{(d-e)}(a+b+e)\ \mu_H^{(d-e)}(e+f+(c-d))\ \mu_H^{(c-f)}(d+f)}.
\end{equation}
The remainder of the proof now proceeds in an identical fashion to that for~$\mathbf{26}\rhd\mathbf{10}$: we show the existence of an $X\in(d+f,\ a+d+e)$ such that
$$\frac{\mu_H^{(d-e)}(X+b-d)}{\mu_H^{(d-e)}(X)} = \frac{\mu_H^{(c-f)}(b+f)}{\mu_H^{(c-f)}(d+f)}$$
by showing using Lagrangian means, that 
$$\frac{\mu_H^{(d-e)}(a+d+e+(b-d))}{\mu_H^{(d-e)}(a+d+e)} < \frac{\mu_H^{(c-f)}(b+f)}{\mu_H^{(c-f)}(d+f)} < \frac{\mu_H^{(d-e)}(b+f)}{\mu_H^{(d-e)}(d+f)},$$
thereby squeezing the desired value between two points on the curve of the monotonically decreasing function $\frac{\mu_H^{(d-e)}(x+(b-d))}{\mu_H^{(d-e)}(x)}$.
We then use lemma~\ref{moa} to relate that to the original question. \end{proof}

\subsubsection{The three conjectural sporadics}

Unfortunately I have been unable to prove $\mathbf{37}\rhd\mathbf{11}$, $\mathbf{43}\rhd\mathbf{11}$ and $\mathbf{49}\rhd\mathbf{11}$: the structure of these three is markedly different from the ones we have just proven, and does not seem to yield to any similar techniques. So we may merely state the following conjecture:

\begin{conjecture}\label{ohno}
In the above notation,
\begin{eqnarray*}
\mathbf{37}\rhd\mathbf{11}:\ \ \left( \begin{array}{ccc}a & c & f\\b & d & e\end{array}\right) & \rhd &\left( \begin{array}{ccc}a & b & d\\f & c & e\end{array}\right)\\
\mathbf{43}\rhd\mathbf{11}:\ \ \left( \begin{array}{ccc}a & d & e\\b & c & f\end{array}\right) & \rhd &\left( \begin{array}{ccc}a & b & d\\f & c & e\end{array}\right)\\
\mathbf{49}\rhd\mathbf{11}:\ \ \left( \begin{array}{ccc}a & d & f\\b & c & e\end{array}\right) & \rhd &\left( \begin{array}{ccc}a & b & d\\f & c & e\end{array}\right).
\end{eqnarray*}
\end{conjecture}

As mentioned in the introduction we shall collectively refer to the above three relations together with the corollary relation~$\mathbf{31}\rhd\mathbf{11}$ as~$\mathbf{C4}$. Also recall the definition of the \emph{binary entropy function}~$h(x)=H(x)+H(1-x)$.
One fascinating result of our numerical work - which to some extent highlights the unusual nature of these four relations -  is that if we simply substitute~$h$ for~$H$ then we obtain a partial order which shares all~826 relations which we have proven for~$H$, together with seven other relations, but the~$\mathbf{C4}$ are broken as may easily be shown by example. Moreover they are broken around fifty percent of the time. So somehow the extra symmetry of the binary entropy function, as opposed to the simple entropy function, wipes out precisely these four relations. One might hope that such a schism in behaviour would point the way to a proof of the conjecture above, although I have been unable to find one: in particular because $h$ does not lend itself to analysis by the methods of this paper (for example, lemma~\ref{lollipop} (ii) fails for $h$). There are many other functions which also break the~$\mathbf{C4}$  exclusively out of the~830, including all quadratics: one way to prove the above conjecture would be to show that entropy lies on a continuous manifold of functions well within the family of functions which respect all~830 relations. However such a proof also seems very difficult because of the convoluted nature of the Lagrangian mean functions which are involved (indeed they are often only piecewise defined).

\subsubsection{Summary of the partial order structure}

So we have a total of $262=165+2+90+5$ relations which in some sense are `primitive': there are no duplicates and the list exhausts all possibilities, bearing in mind that three of these are conjectural. In fact as we mentioned in the proof just now, it is easy to check that any pair {\bf not} included in the relations obtained by viewing these~262 as a (nilpotent) adjacency matrix~$A_{\Rtt}$ and then looking at all the powers of~$A_{\Rtt}$, is not able to be a relation by constructing a few simple random samples say on Matlab. Taking the transitive reduction of this larger graph the overall number of covering edges reduces to~186, made up of~115 majorisation relations and~71 pure entropic relations. That is to say, the process of taking the transitive reduction of~$A_{\Rtt}$ factorises~50 relations from the majorisation side and~24 from the entropic side. As mentioned at the beginning of this section these primitive relations give rise to a total of between~826 and~830 relations overall.

This completes the proof of the analytic side of theorem~\ref{analalg}, once we note that the density of the partial order~$\epo$ is given by a number between $\frac{826}{1770}$ and $\frac{830}{1770}$, that is approximately~0.47 as claimed. It remains to outline the algebraic structure of~$\epo$ in the next chapter. We conclude this chapter with a curious fact about the entropic relations.

\subsubsection{An aside: strange factorisations in the no-man's land between majorisation and $\rhd$}

We remark on a phenomenon which arises in the interplay between majorisation and the relation $\rhd$ which perhaps is a clue to delineating the kind of `majorisation versus disorder' behaviour which Partovi explores in~\cite{partovi}.

Adding the `sporadic' entropic relations from theorem~\ref{proveit} to the~90 `pure entropic' relations from corollary~\ref{yep} we obtain a maximal total of 95 relations which are NOT achievable through majorisation. It turns out that the transitive reduction of the (somewhat artificial) graph on 60 nodes whose edges are these 95 relations in fact is identical to the original graph. That is to say, all 95 are covering relations when we consider only the pure entropic relations (ie no majorisation). Curiously however when the majorisation relations are added in, there are many cases where an entropic edge ceases to be a covering relation and factors through a majorisation plus an entropic, so we have the following strange situation for right coset representatives $L,M,N\in {\Rtt}$:
$$L\succ M\rhd N\Rightarrow L\rhd N{\ \rm but\ }L\not\succ N {\rm\ \ \ \ \ !!}$$
An example of this occurs if we set $L=\mathbf{31}=\left( \begin{array}{ccc}a & c & e\\b & d & f\end{array}\right)$, $M=\mathbf{43}=\left( \begin{array}{ccc}a & d & e\\b & c & f\end{array}\right)$ and $N=\mathbf{10}=\left( \begin{array}{ccc}a & b & d\\e & f & c\end{array}\right)$. Then as is easy to check using the conditions in theorem~\ref{cripes} and the discussion preceding proposition~\ref{crikey}, $L\succ M\rhd N$ (which implies $L\rhd N$ by the transitivity of $\rhd$ and proposition~\ref{majmcmaj}) but $L\not\succ N$.

Similarly a kind of `inverse' situation also occurs - albeit less frequently - namely
$$R\rhd S\succ T\Rightarrow R\rhd T{\ \rm but\ }R\not\succ T.$$
For completeness we mention an example of this too: take $R=\mathbf{5}=\left( \begin{array}{ccc}a & b & c\\f & d & e\end{array}\right)$, $S=\mathbf{11}=\left( \begin{array}{ccc}a & b & d\\f & c & e\end{array}\right)$ and $T=\mathbf{35}=\left( \begin{array}{ccc}a & c & e\\f & b & d\end{array}\right)$.

To get some insight into this we need to show to what extent the two relations $\succ$ and $\rhd$ are the same. Recall from proposition~\ref{majmcmaj} that majorisation implies $\rhd$, but not conversely: for when $M\succ N$ we know from the fact that entropy is a Schur-concave function that $H(\mathbf{r}(M))<H(\mathbf{r}(N))$ and $H(\mathbf{c}(M))<H(\mathbf{c}(N))$. Hence the terms from $N$ entirely dominate those from $M$, giving us the entropic relation $M\rhd N$. We now explore the extent to which the converse might be true.

The relation $M\rhd N$ when $M\not\succ N$ may be thought of as a tug-of-war between the entropy differential of the row vectors $\mathbf{r}(M)$ and $\mathbf{r}(N)$, and that of the column vectors $\mathbf{c}(M)$ and $\mathbf{c}(N)$. In principle it would seem that either column entropy or row entropy could win the tug-of-war - and indeed each of these situations occurs in examples. However it turns out for any fixed pair $M,N$ where $M\rhd N$ that \it a priori \rm either the column vectors always dominate, or the row vectors always dominate.

\begin{proposition}\label{scrapie}
Let $M,N\in\Mtt$. If $M\rhd N$ then \it a priori \rm  either $\mathbf{r}(M)\succ\mathbf{r}(N)$ or $\mathbf{c}(M)\succ\mathbf{c}(N)$.
\end{proposition}

We remark that the differential (row or column) which does \bf not \rm have a majorisation relation acts as a kind of `swing' factor: it may be positive \bf or \rm negative in many instances - obviously if it is always positive then we have $M\succ N$ - but there are also many examples where the other factor \bf ALWAYS \rm has the opposite sign, but never gets large enough to outweigh the effect of the majorisation: indeed this `other' factor majorises the other way, giving us indeed a tug-of-war. For an example of this look at any instance of type~B: we automatically have row sum majorisation in the same direction as the entropic relation (see below), and it is immediate that column sum majorisation goes in the opposite direction.
We should also note that there are~30 pairs of matrix classes $M,N$ where neither $\mathbf{r}(M)\succ\mathbf{r}(N)$ nor $\mathbf{c}(M)\succ\mathbf{c}(N)$ (nor indeed either of the converses) - that is to say, they have no \it a priori \rm relations even on the level of row or column sum vectors. By proposition~\ref{scrapie} there can be no entropic relations between such matrices. Furthermore, none of these~30 examples may be realised by a single transposition: indeed they all involve changes in all three columns and in both rows. In other words they must in general have no common coordinates between the row sum vectors, nor any between the column sum vectors. We list them for reference in appendix~\ref{stasi}.

\begin{proof}
First we note that if $M\succ N$ then the result is known by definition. Furthermore, since majorisation is transitive it follows that if we know the result to be true for an entropic relation $M\rhd N$ and if $P\succ M$ or $N\succ Q$ then we know the result would be true for $P\rhd N$ or $M\rhd Q$. So we need only focus on entropic relations which cannot be factored into any product involving a majorisation step. We consider first of all the~90 relations arising from theorem~\ref{cripes}: that is to say, those which arise from a single `diagonal' transposition $M\mapsto M^\tau$ where $\tau=(\alpha,\beta)$ and as above we represent the matrices by $M = \left( \begin{array}{ccc}\alpha & x & y\\u & \beta & v\end{array}\right)$ and $M^\tau=\left( \begin{array}{ccc}\beta & x & y\\u & \alpha & v\end{array}\right)$. Now `most' of these relations are of the form $M^\tau\rhd M$ - what we referred to as type~A in the proof of theorem~\ref{cripes} above - and we see immediately that the vector of column sums of $M^\tau$ must majorise that of $M$ by virtue of our constant assumptions that $\alpha>\beta$ and $x>u$. So we are done for all type~A entropic relations which arise from a single transposition. For type~B we note again from theorem~\ref{cripes} that a necessary and sufficient condition is $y>x>u>v$, which implies in particular that $x+y>u+v$, which together with $\alpha>\beta$ means that the vector of row sums of $M$ must majorise that of $M^\tau$. 

So it remains to show that the proposition holds for the sporadic~5 relations of theorem~\ref{proveit}, and that it holds when we compose successive entropic relations. The former is easy to show directly (in each of the five sporadics $M\rhd N$ it is the case that $\mathbf{c}(M)\succ\mathbf{c}(N)$). That the proposition holds under composition of relations is obvious (by the transitivity of majorisation) when we consider a sequence of two or more type~A relations and/or sporadic relations; or indeed if we were to consider a sequence consisting only of type~B relations. So the only issue is what happens when we compose a type~B with a sporadic or with a type~A.

The sporadic relations are easy to deal with: recall from appendix~\ref{stasi} the~15 type~B relations. Comparing this list with the list of the sporadic instances in theorem~\ref{proveit} we see that only the following sequences can occur between the two sets:
$\mathbf{15}\rhd\mathbf{10}\rhd\mathbf{60}$,  $\mathbf{15}\rhd\mathbf{10}\rhd\mathbf{24}$,  $\mathbf{26}\rhd\mathbf{10}\rhd\mathbf{60}$, and $\mathbf{26}\rhd\mathbf{10}\rhd\mathbf{24}$.
In particular there are no relations of the form type~B followed by a sporadic.
Considering each in turn we are able to show directly (using say the graph~GR2) that the column sum vectors of the left-hand sides always majorise those of the right-hand sides, hence proving the claim. Indeed the middle two relations $\mathbf{15}\rhd\mathbf{24}$ and $\mathbf{26}\rhd\mathbf{60}$ actually exhibit full majorisation.

The claim for the composition of type~B with type~A follows from a similar case-by-case analysis of the instances where they `match up' (ie where we have a type~A relation~$X\rhd Y$ followed by a type~B relation~$Y\rhd Z$, and vice-versa), using the matrix of~830 relations referred to above. We omit the details.
\end{proof}

\newpage

\section{A purely algebraic construction of the entropic partial order $\epo$}\label{algview}

So we have our partial order~$\epo$ which has been defined entirely in terms of the entropy function. In this next section we shall briefly describe a combinatorial or algebraic construction which presupposes nothing about entropy but whose derivation mimics the case-by-case constructions of proposition~\ref{crikey} and theorem~\ref{cripes}.
Unfortunately I have not been able to find a more natural expression for these relations than this: it is tantalisingly close to a closed form but it seems always to be burdened with some `exceptional' relations which must be subtracted, no matter how they are phrased. 

When we use another strictly convex or strictly concave function~$f$ instead of entropy and define a kind of~$f$-CMI by substituting~$f$ for~$H$ in the definitions, then it is these exceptions which come into play: the coefficients of the summands in~(\ref{ether}) will change depending upon the curvature properties of~$f$, yielding new partial orders. Indeed by studying the simple family of functions~$\{\ \pm x^p\ :\ p\in\RR\ \}$ we are able to construct functions which `tune into' or `tune out of' various components of the partial order~$\epo$, yielding a phenomenon akin to that of the family of Renyi entropies on vectors which approximates Shannon entropy near~1. For example, it is easy to show that the equivalent conditions to theorem~\ref{cripes} for $f(x)=-x^2$ are that type~A occurs if and only if $v>y$, and type~B occurs iff $y>v$; moreover together with the same majorisation relations as for $f=H$ these generate \emph{all} of the~1184 relations which hold for this $f$. Perhaps the most curious fact is that just like the binary entropy function~$h$ defined in the last section, the only relations which are actually \emph{broken} from~$\epo$ in going from $H$ to $f$ are those we have called~$\mathbf{C4}$. As mentioned in the introduction, this is a vast topic for further study.


We say a quick word on the process of finding this algebraic description, which to some extent ties in with the statement of theorem~\ref{analalg}. The `shape' of the group ring elements below was discovered by considering the image of the right coset space~$K\backslash G$ under some of the outer automorphisms of~$G$: namely those which send~$K$ to a parabolic subgroup. There are six parabolic subgroups which are isomorphic to~$K$: $\langle(1,2),(2,3),(4,5)\rangle$, $\langle(1,2),(2,3),(5,6)\rangle$, $\langle(1,2),(3,4),(4,5)\rangle$, $\langle(1,2),(4,5),(5,6)\rangle$,  $\langle(2,3),(3,4),(5,6)\rangle$, $\langle(2,3),(4,5),(5,6)\rangle$, and for each one there exist several outer automorphisms which map $K$ onto it. Choose any such~$J$ and a corresponding outer automorphism $\zeta$. The right coset space~$J\backslash G$ is isomorphic as a~$G$-set to~$K\backslash G$. The image under~$\zeta$ of each matrix class forms a kind of pyramid, with the row- and column-swap equivalences being transformed into equivalences between the positions of a singleton, a pair and a triple of probabilities. Relations between these pyramids turn out to be much easier to visualise than those between matrices, and the (almost-) cyclic structure of our group ring element~$\eta_{\tau,{\bf cyc}}$ below was much more apparent in that form.

\subsection{The abstract combinatorial construction}

Let~$G=\Ssix$ be the symmetric group on the set of six elements~$\{1,2,3,4,5,6\}$. If~$\sigma\in\Ssix$ acts by sending~$i$ to~$\sigma(i)$ then one way of representing~$\sigma$ is to write it as the ordered~$6$-tuple $[\sigma(1),\sigma(2),\sigma(3),\sigma(4),\sigma(5),\sigma(6)].$ 
On the other hand we shall also represent elements of~$G$ in standard cycle notation: as in the rest of the paper, elements are understood to act on the \bf left \rm. That is to say for example that the product~$(1,2)(2,3)$ is equal to~$(1,2,3)$ rather than to~$(1,3,2)$. Define~$K$ to be the subgroup of~$\Ssix$ generated by the elements~$(1,4)(2,6)(3,5)$ and~$(1,6,2,4,3,5)$. Then~$K$ is isomorphic to the dihedral group of order~12.
The reason for choosing this particular subgroup is that when the vectors are arranged in the~$2\times3$-matrix form, left multiplication by this subgroup gives exactly the row- and column-swap operations under which~CMI is invariant: this is clearer if we choose the more obvious generators~$(1,2)(4,5),\ (1,3)(4,6),\ (1,4)(2,5)(3,6)$.

The right coset space~$K\backslash G$ contains~60 elements and may be made into a right module for the action of the group ring~$\ZZ[G]$ by taking the free abelian group whose generators are the right cosets~$K\sigma$ of~$K\backslash G$. Let~$\Id$ denote the multiplicative identity element of~$\ZZ[G]$, which is identified in the usual way with~$1\cdot1_G$ where~$1_G$ is the identity element of~$G$ and~$1$ represents the integer~1. 
Let $\tau=(\alpha,\beta)$ be any transposition in~$G$, and let $\{r,s,t,u\}=\{1,2,3,4,5,6\}\setminus\{\alpha,\beta\}$ represent the four elements left after removing $\alpha$ and $\beta$. Assume that we have ordered them so that $r>s>t>u$. Let~$\psi_\tau = (r,s),\ \chi_\tau = (s,t)$ and~$\gamma_\tau = (\alpha,u,t)(\beta,s,r).$
Let~$\mu_\tau$ be~$(\alpha,\beta)(r,t)(s,u)$, the unique involution which fixes~$\gamma_\tau$ and which interchanges~$(r,s)$ with~$(t,u)$.
Finally, define~$\sigma_\tau$ to be any one of the~12 elements of~$G$ which take~$[1,2,3,4,5,6]$ into the right~$K$-coset of the permutation~$[\alpha,r,s,\beta,u,t]$ by right multiplication.

Using the same notation for group ring elements as for their counterparts in~$G$, with coefficients assumed to be~1 unless otherwise stated, let
$$\eta_{\tau,{\bf horiz}} = (\Id+\psi_\tau)(\Id+\psi_\tau^{\mu_\tau})(\Id+\chi_\tau)-(\Id+\psi_\tau\psi_\tau^{\mu_\tau})\chi_\tau,$$
which upon expansion has six terms, and let
$$\eta_{\tau,{\bf cyc}} = \sigma_\tau(\Id+\gamma_\tau+\gamma_\tau^2)(\Id+\psi_\tau)(\Id+\chi_\tau)-\sigma_\tau\gamma_\tau^2\psi_\tau\chi_\tau,$$
which has eleven terms. Finally define the group ring element
\begin{equation}\label{ether}
\eta_\tau =\left(\eta_{\tau,{\bf horiz}}+\eta_{\tau,{\bf cyc}}\right) (\tau-1),
\end{equation}
which therefore has a total of~17 terms of the form~$\mathfrak{z}(\tau-1)$ for some~$\mathfrak{z}$ representing an element $z\in G$.

\begin{definition}\label{reflepo}
Let $\tau$ run over the~15 transpositions in~$G$.
Define a binary relation $\btr$ on $K\backslash\Ssix$ by letting each summand of each $\eta_\tau$ of the form $\mathfrak{z}(\tau-1)$ represent a relation of the form
$$K\mathfrak{z} \btr K\mathfrak{z}\tau.$$
This yields $15*17=255$ binary relations.
\end{definition}

\begin{theorem}
 The transitive closure of the relations~$\btr$ just defined together with the five \emph{sporadic} relations of theorem~\ref{proveit}, is identical to~$\epo$.
\end{theorem}

\begin{proof}
Take the relations from definition~\ref{reflepo} and theorem~\ref{proveit} and generate their transitive reduction: it is identical to that of~$\epo$ as per appendix~\ref{transco}.
\end{proof}

We may define such an element~$\eta_\tau$ for any of the~15 transpositions in~$G$; or we could equally well take a starting transposition $\tau$ arbitrarily and then `navigate' between all of its conjugates by using only \emph{adjacent} transpositions $\kappa$ which \emph{share a common element with $\tau$}. That is to say,~$\kappa=(\alpha,\alpha\pm1)$ or~$\kappa=(\beta\pm1,\beta)$ with the possibilities obviously constrained by where~$\alpha,\beta$ lie in the set~$\{1,2,3,4,5,6\}$. Denoting by~$g^\kappa$ as usual conjugation of~$g\in G$ by~$\kappa$ we then define $\sigma_\kappa=\sigma_\tau\kappa$, $\psi_\kappa=\psi_\tau^\kappa$, $\chi_\kappa=\chi_\tau^\kappa$, $\gamma_\kappa=\gamma_\tau^\kappa$, $\mu_\kappa=\mu_\tau^\kappa$ and we get the same outcome for~$\eta_\tau$ as we would have done with the direct definitions above. \emph{So it is possible to generate inductively \emph{all} of the~255 relations from one starting point}. The adjacency and common element conditions for~$\kappa$ are necessary because they preserve the rigidity of the orderings~$\alpha^\kappa>\beta^\kappa$ and $r^\kappa>s^\kappa>t^\kappa>u^\kappa$.

All of this raises an intriguing question. Does~$\epo$ correspond to any of the well-known orders on quotients of the symmetric group? The na\"ive answer is no: our partial order is `complicated' in the sense that it is not properly graded: many covering relations have length~$\geq2$ rather than just~$1$ as with the inherited Bruhat orders on the parabolic quotients of the symmetric group from classical Lie algebra theory. So the answer to what~$\epo$ `is' may lie in the more general framework of \it generalised Bruhat quotients\rm~\cite{bjorner}. 

\subsection{The unique involution $\xi$ of the entropic partial order $\epo$}\label{epoi}

Having completed the proof of theorem~\ref{analalg} it remains just to make some final observations about the internal structure of~$\epo$ which arise when one considers whether its graph has any symmetry. Consider the `maximal' involution in the Bruhat order~\cite{bjorner} which is $\omega=(1,6)(2,5)(3,4)\in K$ in the usual cycle notation, and define $\xi_\omega\in\Aut(G)$ to be the unique element of the automorphism group whose action is given by conjugation by $\omega$. We prove here a structure theorem for the graph of the entropic poset $\epo$ on the elements of $K\backslash G$.

\begin{theorem}\label{karma}
$\xi_\omega$ is the unique automorphism of $G$ which respects the entropic partial order $\epo$ on $K\backslash G$.
\end{theorem}
In other words, $\xi_\omega$ induces a graph automorphism of the directed graph on~60 nodes with~186 edges which is conjecturally the graph of covering relations of~$\epo$. Moreover if we ignore the~3 covering relations contained in the unproven relations~$\mathbf{C4}$, this involution still induces an automorphism of the graph of the remaining~183 relations. See appendix~\ref{mxs} for the details of these directed graphs.

\begin{proof}
\it `Analytical': \rm In theorems~\ref{cripes} and~\ref{proveit} we derived from first principles the set of relations which arise only from the binary relation~$\rhd$. In proposition~\ref{crikey} (and see also corollary~\ref{cannotprove}) we explored those relations which arise from majorisation and saw that they are subsumed under the first set. This gave a directed graph on 60 nodes with 262 edges, whose covering relations boil down to~186 edges on the~60 nodes:~$\epo$ is defined to be the transitive closure of these covering relations. Feeding the adjacency matrix of this graph into the program SAGE  (www.sagemath.org) gave us a graph automorphism group~$\{1,\kappa\}$ of order~2, which fixes~16 nodes and acts as an involution on the other~44, splitting them into~22 orbits of~2 matrix classes each. We should also mention that we confirmed the uniqueness of the graph automorphism result using~SAUCY (http://vlsicad.eecs.umich.edu/BK/SAUCY/).

To discover to which (if any) automorphism \emph{of the group~$G$} this graph automorphism~$\kappa$ might correspond we proceeded as follows. The normalizer~$\mathbf{N}_G(K)$ of~$K$ in~$G$ is just~$K$ itself, and no outer automorphism of~$G$ can fix~$K$: consider for example the row-swap element~$(1,4)(2,5)(3,6)$ which must map under any non-trivial outer automorphism to a single transposition~\cite[chapter~7]{rotman}. But there are no single transpositions in~$K$. So the only possible candidates to give by conjugation an (inner) automorphism of~$G$ which preserves the structure of~$K\backslash G$ are the elements of~$K$ itself. Of these only~$\omega$ respects the binary relation~$\rhd$ in every instance (we used the computer program~GAP (www.gap-system.org) to check this, using orbit sizes). So in fact~$\kappa=\xi_\omega$ as claimed. 

\it `Algebraic': \rm once we know the individual relations constructed in definition~\ref{reflepo} we are also able to verify algebraically that conjugation by~$\omega$ swaps these relations among themselves modulo equivalence by left multiplication by elements of~$K$, leaving the total structure unaltered.
\end{proof}

Finally we make a few comments on why this involution preserves the single-transpositional relations within the partial order, this time from a purely theoretical point of view. That $\xi_\omega$ respects majorisation follows from proposition~\ref{mipelice} and the observation that the action of $\xi_\omega$ on figures~\ref{GR2} and~\ref{GR3} is to reflect them in a horizontal line passing through the centre of each: hence the property of lying on a path joining two nodes is unaltered by the action of~$\xi_\omega$.
It is also possible to show directly that~$\xi_\omega$ respects~$\rhd$ relations separated by a single transposition, as follows. Given any $g\in G$, denote by $g^{\xi_\omega}$ the image $\xi_\omega(g)$ of $g$ under the inner automorphism $\xi_\omega$, or in other words $g^{\xi_\omega}=\omega g \omega^{-1}$.

\begin{proposition}\label{respect}
Suppose $\sigma\rhd\sigma\tau$ for some transposition $\tau$. Then
$$\sigma^{\xi_\omega}\rhd\sigma^{\xi_\omega}\tau^{\xi_\omega}.$$
\end{proposition}

\begin{proof}
The easiest way to approach this is to use again the criteria from theorem~\ref{cripes} on pairs of matrix classes. For any letter~$z$ in the set of six letters acted upon by~$G$ let us write~$\bar{z}$ for its image under~$\omega$: so for example~$\bar{a}=f$, etc. Since~$\omega\in K$ and~$\omega^{-1}=\omega$ it follows that the impact of conjugation by~$\omega$ upon a right coset~$K\sigma$ is the same as that of right multiplication by~$\omega$, which in matrix format means we simply replace~$z$ with~$\bar{z}$ everywhere.
So the image under~$\xi$ of the matrix class represented by~$M=\left( \begin{array}{ccc}\alpha & x & y\\u & \beta & v\end{array}\right)$ will be
$$\bar{M}=\left( \begin{array}{ccc}\bar{\alpha} & \bar{x} & \bar{y}\\\bar{u} & \bar{\beta} & \bar{v}\end{array}\right).$$
Now $\xi_\omega$ reverses all size relations and so~$\alpha>\beta$,~$x>u$ become~$\bar{\alpha}<\bar{\beta}$ and~$\bar{x}<\bar{u}$. 
Furthermore the transposition $\tau=(\alpha,\beta)$ becomes ${\tau^{\xi_\omega}}=(\bar{\beta},\bar{\alpha})$.
Putting the matrix~$\bar{M}$ back into the form in the hypotheses of theorem~\ref{cripes} requires that we choose as a representative of the same class~$\bar{M}$ instead:
$$M'=\left( \begin{array}{ccc}\bar{\beta} & \bar{u} & \bar{v}\\\bar{x} & \bar{\alpha} & \bar{y}\end{array}\right).$$
We need to show that $M\rhd M^\tau$ implies that~$M'\rhd {M'}^{\tau^{\xi_\omega}}$ and that $M^\tau\rhd M$ implies~${M'}^{\tau^{\xi_\omega}}\rhd M'$.
Looking again at theorem~\ref{cripes} we see that the necessary and sufficient conditions for type~A and type~B relations give
$$M^\tau\rhd M\Longleftrightarrow v>y \Longleftrightarrow \bar{y}>\bar{v} \Longleftrightarrow {M'}^{\tau^{\xi_\omega}}\rhd M',$$
and
$$M\rhd M^\tau\Longleftrightarrow y>x>u>v\Longleftrightarrow\bar{v}>\bar{u}>\bar{x}>\bar{y}\Longleftrightarrow M'\rhd {M'}^{\tau^{\xi_\omega}}$$
This completes the proof.
\end{proof}

\newpage

\section{Appendices}
\appendix
\section{The matrix class representatives in $\Rtt$}\label{frog}

We list the matrix representatives in~$\Rtt$ in lexicographic order together with the lexicographic enumeration we have used throughout the paper when referring to them, alongside in each case the element~$\sigma\in G=\Ssix$ in cycle notation which represents the appropriate permutation of the fiducial matrix~$\left( \begin{array}{ccc}a & b & c \\d & e & f\end{array}\right)$ which we have chosen to represent the identity~$()\in G$. Note that each~$\sigma$ is only chosen up to \it left \rm multiplication by an element of~$K$. Also, since we have chosen to represent the matrices with~$a$ in the top left-hand corner and with decreasing top row, the set of representative cycles displayed is effectively a copy of~$\Sfive$ modulo a subgroup of order~2.\\

\begin{tabular}{| l l l | l l l |l l l | l l l |}
\hline
 $\mathbf{1}:$ & $\left( \begin{array}{ccc}a & b & c\\d & e & f\end{array}\right)$, & $\mathbf{()}$ &
 $\mathbf{2}:$ & $\left( \begin{array}{ccc}a & b & c\\d & f & e\end{array}\right)$, & $\mathbf{(56)}$ &
 $\mathbf{3}:$ & $\left( \begin{array}{ccc}a & b & c\\e & d & f\end{array}\right)$, & $\mathbf{(45)}$ & 
 $\mathbf{4}:$ & $\left( \begin{array}{ccc}a & b & c\\e & f & d\end{array}\right)$, & $\mathbf{(465)}$ \\ \hline 
 $\mathbf{5}:$ & $\left( \begin{array}{ccc}a & b & c\\f & d & e\end{array}\right)$, & $\mathbf{(456)}$ & 
 $\mathbf{6}:$ & $\left( \begin{array}{ccc}a & b & c\\f & e & d\end{array}\right)$, & $\mathbf{(46)}$ &
 $\mathbf{7}:$ & $\left( \begin{array}{ccc}a & b & d\\c & e & f\end{array}\right)$, & $\mathbf{(34)}$ & 
 $\mathbf{8}:$ & $\left( \begin{array}{ccc}a & b & d\\c & f & e\end{array}\right)$, & $\mathbf{(34)(56)}$ \\ \hline  
 $\mathbf{9}:$ & $\left( \begin{array}{ccc}a & b & d\\e & c & f\end{array}\right)$, & $\mathbf{(354)}$ & 
 $\mathbf{10}:$ & $\left( \begin{array}{ccc}a & b & d\\e & f & c\end{array}\right)$, & $\mathbf{(3654)}$ & 
 $\mathbf{11}:$ & $\left( \begin{array}{ccc}a & b & d\\f & c & e\end{array}\right)$, & $\mathbf{(3564)}$ & 
 $\mathbf{12}:$ & $\left( \begin{array}{ccc}a & b & d\\f & e & c\end{array}\right)$, & $\mathbf{(364)}$ \\ \hline
 $\mathbf{13}:$ & $\left( \begin{array}{ccc}a & b & e\\c & d & f\end{array}\right)$, & $\mathbf{(345)}$ & 
 $\mathbf{14}:$ & $\left( \begin{array}{ccc}a & b & e\\c & f & d\end{array}\right)$, & $\mathbf{(3465)}$ & 
 $\mathbf{15}:$ & $\left( \begin{array}{ccc}a & b & e\\d & c & f\end{array}\right)$, & $\mathbf{(35)}$ & 
 $\mathbf{16}:$ & $\left( \begin{array}{ccc}a & b & e\\d & f & c\end{array}\right)$, & $\mathbf{(365)}$ \\ \hline 
 $\mathbf{17}:$ & $\left( \begin{array}{ccc}a & b & e\\f & c & d\end{array}\right)$, & $\mathbf{(35)(46)}$ & 
 $\mathbf{18}:$ & $\left( \begin{array}{ccc}a & b & e\\f & d & c\end{array}\right)$, & $\mathbf{(3645)}$ &
 $\mathbf{19}:$ & $\left( \begin{array}{ccc}a & b & f\\c & d & e\end{array}\right)$, & $\mathbf{(3456)}$ & 
 $\mathbf{20}:$ & $\left( \begin{array}{ccc}a & b & f\\c & e & d\end{array}\right)$, & $\mathbf{(346)}$ \\ \hline 
 $\mathbf{21}:$ & $\left( \begin{array}{ccc}a & b & f\\d & c & e\end{array}\right)$, & $\mathbf{(356)}$ & 
 $\mathbf{22}:$ & $\left( \begin{array}{ccc}a & b & f\\d & e & c\end{array}\right)$, & $\mathbf{(36)}$ & 
 $\mathbf{23}:$ & $\left( \begin{array}{ccc}a & b & f\\e & c & d\end{array}\right)$, & $\mathbf{(3546)}$ & 
 $\mathbf{24}:$ & $\left( \begin{array}{ccc}a & b & f\\e & d & c\end{array}\right)$, & $\mathbf{(36)(45)}$ \\ \hline 
 $\mathbf{25}:$ & $\left( \begin{array}{ccc}a & c & d\\b & e & f\end{array}\right)$, & $\mathbf{(243)}$ & 
 $\mathbf{26}:$ & $\left( \begin{array}{ccc}a & c & d\\b & f & e\end{array}\right)$, & $\mathbf{(243)(56)}$ & 
 $\mathbf{27}:$ & $\left( \begin{array}{ccc}a & c & d\\e & b & f\end{array}\right)$, & $\mathbf{(2543)}$ & 
 $\mathbf{28}:$ & $\left( \begin{array}{ccc}a & c & d\\e & f & b\end{array}\right)$, & $\mathbf{(26543)}$  \\ \hline 
 $\mathbf{29}:$ & $\left( \begin{array}{ccc}a & c & d\\f & b & e\end{array}\right)$, & $\mathbf{(25643)}$ & 
 $\mathbf{30}:$ & $\left( \begin{array}{ccc}a & c & d\\f & e &b\end{array}\right)$, & $\mathbf{(2643)}$ &
 $\mathbf{31}:$ & $\left( \begin{array}{ccc}a & c & e\\b & d & f\end{array}\right)$, & $\mathbf{(2453)}$ & 
 $\mathbf{32}:$ & $\left( \begin{array}{ccc}a & c & e\\b & f & d\end{array}\right)$, & $\mathbf{(24653)}$ \\ \hline 
 $\mathbf{33}:$ & $\left( \begin{array}{ccc}a & c & e\\d & b & f\end{array}\right)$, & $\mathbf{(253)}$ & 
 $\mathbf{34}:$ & $\left( \begin{array}{ccc}a & c & e\\d & f & b\end{array}\right)$, & $\mathbf{(2653)}$ & 
 $\mathbf{35}:$ & $\left( \begin{array}{ccc}a & c & e\\f & b & d\end{array}\right)$, & $\mathbf{(253)(46)}$ & 
 $\mathbf{36}:$ & $\left( \begin{array}{ccc}a & c & e\\f & d &b\end{array}\right)$, & $\mathbf{(26453)}$ \\ \hline
 $\mathbf{37}:$ & $\left( \begin{array}{ccc}a & c & f\\b & d & e\end{array}\right)$, & $\mathbf{(24563)}$ & 
 $\mathbf{38}:$ & $\left( \begin{array}{ccc}a & c & f\\b & e & d\end{array}\right)$, & $\mathbf{(2463)}$ & 
 $\mathbf{39}:$ & $\left( \begin{array}{ccc}a & c & f\\d & b & e\end{array}\right)$, & $\mathbf{(2563)}$ & 
 $\mathbf{40}:$ & $\left( \begin{array}{ccc}a & c & f\\d & e & b\end{array}\right)$, & $\mathbf{(263)}$  \\ \hline 
 $\mathbf{41}:$ & $\left( \begin{array}{ccc}a & c & f\\e & b & d\end{array}\right)$, & $\mathbf{(25463)}$ & 
 $\mathbf{42}:$ & $\left( \begin{array}{ccc}a & c & f\\e & d &b\end{array}\right)$, & $\mathbf{(263)(45)}$ & 
 $\mathbf{43}:$ & $\left( \begin{array}{ccc}a & d & e\\b & c & f\end{array}\right)$, & $\mathbf{(24)(35)}$ & 
 $\mathbf{44}:$ & $\left( \begin{array}{ccc}a & d & e\\b & f & c\end{array}\right)$, & $\mathbf{(24)(365)}$ \\ \hline 
 $\mathbf{45}:$ & $\left( \begin{array}{ccc}a & d & e\\c & b & f\end{array}\right)$, & $\mathbf{(2534)}$ & 
 $\mathbf{46}:$ & $\left( \begin{array}{ccc}a & d & e\\c & f & b\end{array}\right)$, & $\mathbf{(26534)}$ & 
 $\mathbf{47}:$ & $\left( \begin{array}{ccc}a & d & e\\f & b & c\end{array}\right)$, & $\mathbf{(25364)}$ & 
 $\mathbf{48}:$ & $\left( \begin{array}{ccc}a & d & e\\f & c &b\end{array}\right)$, & $\mathbf{(264)(35)}$ \\ \hline 
 $\mathbf{49}:$ & $\left( \begin{array}{ccc}a & d & f\\b & c & e\end{array}\right)$, & $\mathbf{(24)(356)}$ & 
 $\mathbf{50}:$ & $\left( \begin{array}{ccc}a & d & f\\b & e & c\end{array}\right)$, & $\mathbf{(24)(36)}$ & 
 $\mathbf{51}:$ & $\left( \begin{array}{ccc}a & d & f\\c & b & e\end{array}\right)$, & $\mathbf{(25634)}$ & 
 $\mathbf{52}:$ & $\left( \begin{array}{ccc}a & d & f\\c & e & b\end{array}\right)$, & $\mathbf{(2634)}$  \\ \hline 
 $\mathbf{53}:$ & $\left( \begin{array}{ccc}a & d & f\\e & b & c\end{array}\right)$, & $\mathbf{(254)(36)}$ & 
 $\mathbf{54}:$ & $\left( \begin{array}{ccc}a & d & f\\e & c &b\end{array}\right)$, & $\mathbf{(26354)}$ & 
 $\mathbf{55}:$ & $\left( \begin{array}{ccc}a & e & f\\b & c & d\end{array}\right)$, & $\mathbf{(24635)}$ & 
 $\mathbf{56}:$ & $\left( \begin{array}{ccc}a & e & f\\b & d & c\end{array}\right)$, & $\mathbf{(245)(36)}$  \\ \hline 
 $\mathbf{57}:$ & $\left( \begin{array}{ccc}a & e & f\\c & b & d\end{array}\right)$, & $\mathbf{(25)(346)}$ & 
 $\mathbf{58}:$ & $\left( \begin{array}{ccc}a & e & f\\c & d & b\end{array}\right)$, & $\mathbf{(26345)}$ & 
 $\mathbf{59}:$ & $\left( \begin{array}{ccc}a & e & f\\d & b & c\end{array}\right)$, & $\mathbf{(25)(36)}$ & 
 $\mathbf{60}:$ & $\left( \begin{array}{ccc}a & e & f\\d & c &b\end{array}\right)$, & $\mathbf{(2635)}$ \\ \hline 
\end{tabular}

\newpage

\section{Matrices referred to in the text}\label{mxs}

Here we reproduce the often rather large matrices which are referred to in the text in the course of certain calculations, but which would make the main body of the paper too cumbersome if they appeared there.

\subsection{GR2, GR3 and the matrix of all majorisation relations}
First, we adopt as always the lexicographic ordering of the elements of GR2 (ie $a+b,a+c,\ldots,e+f$) and then the matrix $\mathbf{M_2}$ is as follows:
$$\mathbf{M_2}=\left(\begin{array}{ccccccccccccccc}
0&1&0&0&0&0&0&0&0&0&0&0&0&0&0\\
0&0&1&0&0&1&0&0&0&0&0&0&0&0&0\\
0&0&0&1&0&0&1&0&0&0&0&0&0&0&0\\
0&0&0&0&1&0&0&1&0&0&0&0&0&0&0\\
0&0&0&0&0&0&0&0&1&0&0&0&0&0&0\\
0&0&0&0&0&0&1&0&0&0&0&0&0&0&0\\
0&0&0&0&0&0&0&1&0&1&0&0&0&0&0\\
0&0&0&0&0&0&0&0&1&0&1&0&0&0&0\\
0&0&0&0&0&0&0&0&0&0&0&1&0&0&0\\
0&0&0&0&0&0&0&0&0&0&1&0&0&0&0\\
0&0&0&0&0&0&0&0&0&0&0&1&1&0&0\\
0&0&0&0&0&0&0&0&0&0&0&0&0&1&0\\
0&0&0&0&0&0&0&0&0&0&0&0&0&1&0\\
0&0&0&0&0&0&0&0&0&0&0&0&0&0&1\\
0&0&0&0&0&0&0&0&0&0&0&0&0&0&0\end{array}
\right).$$

Again with the lexicographic ordering of the elements of GR3 (ie $a+b+c,a+b+d,\ldots,d+e+f$) the matrix $\mathbf{M_3}$ is as follows:
$$\mathbf{M_3}=\left(\begin{array}{cccccccccccccccccccc}
0&1&0&0&0&0&0&0&0&0&0&0&0&0&0&0&0&0&0&0\\
0&0&1&0&1&0&0&0&0&0&0&0&0&0&0&0&0&0&0&0\\
0&0&0&1&0&1&0&0&0&0&0&0&0&0&0&0&0&0&0&0\\
0&0&0&0&0&0&1&0&0&0&0&0&0&0&0&0&0&0&0&0\\
0&0&0&0&0&1&0&0&0&0&1&0&0&0&0&0&0&0&0&0\\
0&0&0&0&0&0&1&1&0&0&0&1&0&0&0&0&0&0&0&0\\
0&0&0&0&0&0&0&0&1&0&0&0&1&0&0&0&0&0&0&0\\
0&0&0&0&0&0&0&0&1&0&0&0&0&1&0&0&0&0&0&0\\
0&0&0&0&0&0&0&0&0&1&0&0&0&0&1&0&0&0&0&0\\
0&0&0&0&0&0&0&0&0&0&0&0&0&0&0&1&0&0&0&0\\
0&0&0&0&0&0&0&0&0&0&0&1&0&0&0&0&0&0&0&0\\
0&0&0&0&0&0&0&0&0&0&0&0&1&1&0&0&0&0&0&0\\
0&0&0&0&0&0&0&0&0&0&0&0&0&0&1&0&0&0&0&0\\
0&0&0&0&0&0&0&0&0&0&0&0&0&0&1&0&1&0&0&0\\
0&0&0&0&0&0&0&0&0&0&0&0&0&0&0&1&0&1&0&0\\
0&0&0&0&0&0&0&0&0&0&0&0&0&0&0&0&0&0&1&0\\
0&0&0&0&0&0&0&0&0&0&0&0&0&0&0&0&0&1&0&0\\
0&0&0&0&0&0&0&0&0&0&0&0&0&0&0&0&0&0&1&0\\
0&0&0&0&0&0&0&0&0&0&0&0&0&0&0&0&0&0&0&1\\
0&0&0&0&0&0&0&0&0&0&0&0&0&0&0&0&0&0&0&0\end{array}
\right).$$

The $60\times60$ matrix reflecting all transpositions is now easily generated using the criteria in the text, once we form the sums of powers of these two matrices.

\subsection{Majorisation via transpositions}

The easiest way to give the structure of the~165 single-transposition majorisation relations referred to in proposition~\ref{crikey}, is to represent them as~165 ordered pairs using the numbering in appendix~\ref{frog}. From this set of pairs it is straightforward to rebuild the adjacency matrix of the partial order, namely we just put a `1' in each place whose entry coordinates are given by one of these pairs, and zeroes elsewhere.

The 165 relations are thus:

$(1,2),\ (1,3),\ (2,4),\ (3,4),\ (2,5),\ (3,5),\ (1,6),\ (4,6),\ (5,6),\ (7,8),\ (7,9),\ (4,10),\ (8,10),\ (9,10),\ (8,11),$

$(9,11),\ (6,12),\ (7,12),\ (10,12),\ (11,12),\ (8,14),\ (13,14),\ (13,15),\ (2,16),\ (14,16),\ (15,16),\ (11,17),\ (14,17),\ $

$(15,17),\ (5,18),\ (13,18),\ (16,18),\ (17,18),\ (13,19),\ (7,20),\ (19,20),\ (15,21),\ (19,21),\ (1,22),\ (20,22),\ (21,22),\ $

$(9,23),\ (20,23),\ (21,23),\ (3,24),\ (19,24),\ (22,24),\ (23,24),\ (25,26),\ (9,27),\ (25,27),\ (3,28),\ (26,28),\ (27,28),$

$(11,29),\ (26,29),\ (27,29),\ (5,30),\ (25,30),\ (28,30),\ (29,30),\ (26,32),\ (31,32),\ (15,33),\ (31,33),\ (1,34),\ (7,34),$

$(32,34),\ (33,34),\ (11,35),\ (17,35),\ (29,35),\ (32,35),\ (33,35),\ (6,36),\ (12,36),\ (18,36),\ (30,36),\ (31,36),\ (34,36),$

$(35,36),\ (31,37),\ (25,38),\ (37,38),\ (15,39),\ (33,39),\ (37,39),\ (2,40),\ (8,40),\ (16,40),\ (34,40),\ (38,40),\ (39,40),$

$(9,41),\ (27,41),\ (38,41),\ (39,41),\ (4,42),\ (10,42),\ (14,42),\ (28,42),\ (37,42),\ (40,42),\ (41,42),\ (31,43),\ (25,44),$

$(43,44),\ (13,45),\ (43,45),\ (7,46),\ (44,46),\ (45,46),\ (5,47),\ (11,47),\ (18,47),\ (30,47),\ (35,47),\ (44,47),\ (45,47),$

$(6,48),\ (12,48),\ (17,48),\ (29,48),\ (36,48),\ (43,48),\ (46,48),\ (47,48),\ (31,49),\ (26,50),\ (32,50),\ (49,50),\ (13,51),$

$(49,51),\ (8,52),\ (14,52),\ (50,52),\ (51,52),\ (3,53),\ (9,53),\ (28,53),\ (34,53),\ (50,53),\ (51,53),\ (4,54),\ (10,54),$

$(16,54),\ (27,54),\ (49,54),\ (52,54),\ (53,54),\ (25,55),\ (26,56),\ (55,56),\ (7,57),\ (55,57),\ (8,58),\ (56,58),\ (57,58),$

$(1,59),\ (56,59),\ (57,59),\ (2,60),\ (55,60),\ (58,60),\ (59,60).$

It is easy to verify that the following 30 pairs are not covering relations even just within the pure majorisation framework (eg observe that $(1,6)$ factorises into the product $(1,2)\rightarrow(2,4)\rightarrow(4,6)$):

$(1,6),\  (7,12),\ (13,18)  ,\ (19,24)  ,\ (25,30)  ,\ (11,35)  ,\  (6,36)  ,\ (31,36)  ,\ (15,39)  ,\  (2,40)  ,\  (8,40)  ,\  (9,41)  ,\  (4,42),$

$  (14,42)  ,\   (37,42)  ,\    (5,47)  ,\   (11,47)  ,\   (6,48)  ,\   (12,48)  ,\   (17,48)  ,\   (29,48)  ,\   (43,48)  ,\   (26,50)  ,\    (8,52)  ,\    (3,53)  ,\   (9,53)  ,$

$   (4,54)  ,\ (27,54)  ,\  (49,54)  ,\  (55,60).$

The following 20 additionally will disappear (ie they factorise as a sequence of other relations) once all of the entropic relations are introduced:

$   (5,18),\ (1,22)  ,\  (3,24)  ,\    (5,30)  ,\   (1,34)  ,\   (7,34)  ,\  (32,35)  ,\  (33,35)  ,\ (12,36)  ,\  (16,40)  ,\  (38,40)  ,\  (10,42)  ,\  (44,47)  ,$

$  (45,47)  ,\ (46,48)  ,\  (28,53)  ,\  (51,53)  ,\   (10,54)  ,\   (1,59)  ,\  (2,60).$

As mentioned in the text, the two `exceptional' relations referred to in corollary~\ref{cannotprove} do not give rise to any covering relations; hence we are left with just 115 covering relations in the entropic partial order $\epo$ which arise solely from majorisation.

\subsection{Entropic partial order}\label{stasi}

As in the previous section we shall use ordered pairs (`sparse matrix representation') to give the set of all 90 non-majorisation entropic relations arising from single transpositions as per theorem~\ref{cripes}, viz.:

$  (31,10)  ,\   (43,10)  ,\    (5,11)  ,\   (37,12)  ,\   (49,12)  ,\   (26,14)  ,\ (32,14)  ,\   (25,16)  ,\   (44,16)  ,\    (6,17)  ,\   (50,17)  ,\   (56,17) ,$

$  (12,18)  ,\   (38,18)  ,\   (55,18)  ,\   (31,19)  ,\   (37,19)  ,\   (25,20)  ,\   (38,20)  ,\   (43,21)  ,\   (49,21)  ,\   (16,22)  ,\   (26,22)  ,\   (50,22)  ,$

$   (4,23)  ,\   (44,23)  ,\   (55,23)  ,\   (10,24)  ,\   (32,24)  ,\   (56,24)  ,\   (15,27)  ,\   (33,27)  ,\   (13,28)  ,\   (45,28)  ,\    (6,29)  ,\   (21,29)  ,$

$  (39,29)  ,\   (12,30)  ,\   (19,30)  ,\   (51,30)  ,\ (46,34)  ,\    (5,35)  ,\   (23,35)  ,\   (41,35)  ,\   (52,35)  ,\   (58,35)  ,\ (20,36)  ,\  (57,36)  ,$

$  (45,39)  ,\   (51,39)  ,\   (22,40)  ,\   (52,40)  ,\    (3,41)  ,\   (46,41)  ,\   (57,41)  ,\   (24,42)  ,\   (58,42)  ,\    (1,46)  ,\   (24,47)  ,\   (40,47)  ,$

$  (53,47)  ,\   (60,47)  ,\   (22,48)  ,\   (42,48)  ,\   (54,48)  ,\   (59,48)  ,\   (38,50)  ,\   (44,50)  ,\    (2,52)  ,\   (20,52)  ,\ (46,52)  ,\   (41,53)  ,$

$  (59,53)  ,\   (23,54)  ,\   (60,54)  ,\   (31,55)  ,\   (49,55)  ,\   (37,56)  ,\   (43,56)  ,\   (13,57)  ,\   (51,57)  ,$

$   (4,58)  ,\   (19,58)  ,\   (45,58)  ,\   (15,59)  ,\   (28,59)  ,\ (39,59)  ,\   (10,60)  ,\   (21,60)  ,\   (33,60).$

In the proof of theorem~\ref{cripes} we split these~90 single-transposition entropic relations into two subsets: type~A and type~B. There are only~15 type~B relations (the remainder above are type~A), one corresponding to each of the~15 transpositions $(a,b),\ (a,c)$ etc. We list them here for reference:

$(a,b):\mathbf{28}\rhd\mathbf{59}$;  $(a,c):\mathbf{10}\rhd\mathbf{60}$;  $(a,d):\mathbf{4}\rhd\mathbf{58}$;  $(a,e):\mathbf{2}\rhd\mathbf{52}$;  $(a,f):\mathbf{1}\rhd\mathbf{46}$;  $(b,c):\mathbf{12}\rhd\mathbf{30}$;  $(b,d):\mathbf{6}\rhd\mathbf{29}$;  $(b,e):\mathbf{5}\rhd\mathbf{35}$;  $(b,f):\mathbf{3}\rhd\mathbf{41}$;  $(c,d):\mathbf{5}\rhd\mathbf{11}$;  $(c,e):\mathbf{6}\rhd\mathbf{17}$;  $(c,f):\mathbf{4}\rhd\mathbf{23}$;  $(d,e):\mathbf{12}\rhd\mathbf{18}$;  $(d,f):\mathbf{10}\rhd\mathbf{24}$;  $(e,f):\mathbf{16}\rhd\mathbf{22}$.

Returning to the total set of entropic relations above, we need now to add in the 5 `sporadic' multiple-transposition relations (3 of which are conjectural) from theorem~\ref{proveit}:

$  (15,10),\ (26,10) ,\ (37,11),\ (43,11),\ (49,11) .$

The transitive reduction of this total set of 95 relations is just the set again - that is to say, all 95 relations are covering relations just within the context of `purely entropic' relations. However 24 of them will factorise once we introduce the majorisation relations, as follows:

$  (31,10)  ,\   (37,12)  ,\   (49,12)  ,\   (26,14)  ,\   (25,16)  ,\   (38,18)  ,\   (55,18)  ,\   (31,19)  ,$

$   (25,20)  ,\   (26,22)  ,\   (32,24)  ,\   (15,27)  ,\ (13,28)  ,\   (19,30)  ,\   (51,30)  ,\    (5,35)  ,$

$  (20,36)  ,\   (57,36)  ,\   (22,48)  ,\   (59,48)  ,\   (31,55)  ,\   (13,57)  ,\   (15,59)  ,\   (33,60).$

Notice that all but one of these (namely $(5,35)$ ) are type A.
So we are left with 71 covering `purely entropic' relations (ie which are not ascribable to majorisation), which together with the 115 majorisation relations in the previous section, gives us our complete set of 186 covering relations for the entropic partial order $\epo$. We give this complete set in the next section.

Finally, we list the~30~pairs of matrices mentioned after proposition~\ref{scrapie} where neither row sum majorisation nor column sum majorisation obtain in either direction (so in particular no entropic relation would even be possible):

$\{14,28\},\ \{14,59\},\ \{16,27\},\ \{16,29\},\ \{17,28\},\ \{17,59\},\ \{17,60\},\ \{22,27\},\ \{22,29\},\ \{22,35\},$

$\{22,41\},\ \{22,52\},\ \{22,58\},\ \{23,34\},\ \{23,40\},\ \{23,59\},\ \{23,60\},\ \{24,29\},\ \{24,35\},\ \{24,52\},$

$\{24,58\},\ \{34,58\},\ \{35,59\},\ \{35,60\},\ \{41,59\},\ \{41,60\},\ \{42,47\},\ \{47,54\},\ \{52,59\},\ \{53,58\}.$

\subsection{The entropic partial ordering $\epo$ on $K\backslash G\cong\Rtt$}\label{transco}

Here is the final set of 186 covering relations:

$   (1,2)   ,\    (1,3)   ,\    (2,4)   ,\ (3,4)   ,\    (2,5)   ,\    (3,5)   ,\    (4,6)   ,\    (5,6)   ,\    (7,8)   ,\    (7,9)   ,\    (4,10)  ,\    (8,10)  ,\ $

$   (9,10)  ,\  (15,10)  ,\   (26,10)  ,\   (43,10)  ,\    (5,11)  ,\    (8,11)  ,\    (9,11)  ,\   (37,11)  ,\   (43,11)  ,\   (49,11)  ,\    (6,12)  ,\   (10,12)  ,\ $

$  (11,12)  ,\    (8,14)  ,\   (13,14)  ,\   (32,14)  ,\   (13,15)  ,\    (2,16)  ,\   (14,16)  ,\   (15,16)  ,\   (44,16)  ,\    (6,17)  ,\   (11,17)  ,\   (14,17)  ,\ $

$  (15,17)  ,\   (50,17)  ,\   (56,17)  ,\   (12,18)  ,\   (16,18)  ,\ (17,18)  ,\   (13,19)  ,\   (37,19)  ,\    (7,20)  ,\   (19,20)  ,\   (38,20)  ,\   (15,21)  ,\ $

$  (19,21)  ,\   (43,21)  ,\   (49,21)  ,\   (16,22)  ,\   (20,22)  ,\   (21,22)  ,\   (50,22)  ,\    (4,23)  ,\    (9,23)  ,\   (20,23)  ,\ (21,23)  ,\   (44,23)  ,\ $

$  (55,23)  ,\   (10,24)  ,\   (22,24)  ,\   (23,24)  ,\ (56,24)  ,\   (25,26)  ,\ (9,27)  ,\   (25,27)  ,\   (33,27)  ,\    (3,28)  ,\   (26,28)  ,\   (27,28)  ,\ $

$  (45,28)  ,\    (6,29)  ,\   (11,29)  ,\   (21,29)  ,\ (26,29)  ,\   (27,29)  ,\   (39,29)  ,\   (12,30)  ,\   (28,30)  ,\ (29,30)  ,\   (26,32)  ,\   (31,32)  ,\ $

$  (15,33)  ,\   (31,33)  ,\   (32,34)  ,\   (33,34)  ,\   (46,34)  ,\   (17,35)  ,\ (23,35)  ,\   (29,35)  ,\ (41,35)  ,\   (52,35)  ,\   (58,35)  ,\   (18,36)  ,\ $

$  (30,36)  ,\  (34,36)  ,\   (35,36)  ,\   (31,37)  ,\ (25,38)  ,\   (37,38)  ,\   (33,39)  ,\   (37,39)  ,\   (45,39)  ,\ (51,39)  ,\   (22,40)  ,\   (34,40)  ,\ $

$  (39,40)  ,\   (52,40)  ,\    (3,41)  ,\   (27,41)  ,\   (38,41)  ,\   (39,41)  ,\   (46,41)  ,\   (57,41)  ,\   (24,42)  ,\   (28,42)  ,\ (40,42)  ,\   (41,42)  ,\ $

$  (58,42)  ,\   (31,43)  ,\   (25,44)  ,\   (43,44)  ,\   (13,45)  ,\   (43,45)  ,\    (1,46)  ,\    (7,46)  ,\   (44,46)  ,\   (45,46)  ,\   (18,47)  ,\   (24,47)  ,\ $

$  (30,47)  ,\   (35,47)  ,\   (40,47)  ,\   (53,47)  ,\   (60,47)  ,\ (36,48)  ,\   (42,48)  ,\   (47,48)  ,\   (54,48)  ,\ (31,49)  ,\   (32,50)  ,\   (38,50)  ,\ $

$  (44,50)  ,\   (49,50)  ,\   (13,51)  ,\   (49,51)  ,\    (2,52)  ,\   (14,52)  ,\   (20,52)  ,\   (46,52)  ,\   (50,52)  ,\   (51,52)  ,\   (34,53)  ,\ (41,53)  ,\ $

$  (50,53)  ,\   (59,53)  ,\   (16,54)  ,\   (23,54)  ,\   (52,54)  ,\   (53,54)  ,\   (60,54)  ,\   (25,55)  ,\   (49,55)  ,\   (26,56)  ,\   (37,56)  ,\   (43,56)  ,\ $

$  (55,56)  ,\    (7,57)  ,\   (51,57)  ,\   (55,57)  ,\    (4,58)  ,\    (8,58)  ,\   (19,58)  ,\   (45,58)  ,\   (56,58)  ,\   (57,58)  ,\   (28,59)  ,\ (39,59)  ,\ $
 
$  (56,59)  ,\  (57,59)  ,\   (10,60)  ,\   (21,60)  ,\   (58,60)  ,\   (59,60).$

Set out as an adjacency matrix it represents a directed graph on~60 nodes with~186 edges, and as mentioned in theorem~\ref{karma} this graph has a unique automorphism of order~2 which we called $\xi_\omega$, where $\omega$ is the unique involution $(1,6)(2,5)(3,4)$ of maximal length in the subgroup $K$ of $G=\Ssix$ and $\xi_\omega$ is the inner automorphism of $G$ which is given by conjugation by $\omega$ within $G$. The orbits of $\xi_\omega$ consist of~22 pairs of nodes which are swapped by $\xi_\omega$, together with the remaining 16 nodes which are fixed by its action. We now give the orbits, using the matrix-numbering notation above:

$\{2,3\},\ \{8,9\},\ \{13,25\},\ \{14,27\},\ \{15,26\},\ \{16,28\},\ \{17,29\},\ \{18,30\},\ \{19,55\},\ \{20,57\},\ \{21,56\},\ \{22,59\},$

$\{23,58\},\ \{24,60\},\ \{32,33\},\ \{37,49\},\ \{38,51\},\ \{39,50\},\ \{40,53\},\ \{41,52\},\ \{42,54\},\ \{44,45\}, $

$\{1\}, \{4\}, \{5\}, \{6\}, \{7\}, \{10\}, \{11\}, \{12\}, \{31\}, \{34\}, \{35\}, \{36\}, \{43\}, \{46\}, \{47\}, \{48\}.$

Another way to say this is that conjugation by $\omega$ amounts to an involution $\chi_\omega$ in the symmetric group on $K\backslash G$: that is to say $\chi_\omega\in\mathbf{S}_{60}$ and in cycle notation it has the form:
\begin{eqnarray*}
\chi_\omega & = & (1)(2,3)(4)(5)(6)(7)(8,9)(10)(11)(12)(13,25)(14,27)(15,26)(16,28)\cdot\\
&&\ \ \cdot(17,29)(18,30)(19,55)(20,57)(21,56)(22,59)(23,58)(24,60)(31)(32,33)\cdot\\
&&\ \ \ \ \cdot(34)(35)(36)(37,49)(38,51)(39,50)(40,53)(41,52)(42,54)(43)(44,45)(46)(47)(48).
\end{eqnarray*}

See also figure~\ref{hexcomb}: the blue double-headed dashed lines represent the action of $\chi_\omega$ on the matrix classes of $\Rtt$.

\end{document}